\documentclass[12pt]{amsart}
\usepackage{color}
\usepackage{amssymb}
\usepackage{comment}
\usepackage{pdfpages}
\usepackage{graphicx}
\usepackage{dcolumn}

\usepackage[normalem]{ulem}

\RequirePackage[numbers]{natbib}
\RequirePackage[colorlinks=true, pdfstartview=FitV, linkcolor=blue,
  citecolor=blue, urlcolor=blue]{hyperref}
\RequirePackage{hypernat}
\usepackage{paralist}

\definecolor{darkgreen}{rgb}{0.0,0.39,0.0}

\usepackage{graphicx}
\graphicspath{{./figures/}}
\usepackage{amsfonts}
\usepackage{amsmath}
\usepackage{amsthm}
\usepackage{amssymb}
\usepackage{amsbsy}
\usepackage{epsfig}
\usepackage{fullpage}
\usepackage{natbib, mathrsfs}
\usepackage{verbatim}
\usepackage[latin1]{inputenc}
\usepackage{mhequ}
\usepackage{algorithm}
\usepackage{algpseudocode}
\usepackage{float}
\usepackage[letterpaper,left=1.5in,right=1.5in]{geometry}
\usepackage{enumitem}

\numberwithin{equation}{section}

\def \be{\begin{equs}}
\def \ee{\end{equs}}

\let\HungarianH\H
\def \H{\mathcal{H}}

\newtheorem{theorem}{Theorem}[section]
\newtheorem{lemma}[theorem]{Lemma}
\newtheorem{corollary}[theorem]{Corollary}

\newtheorem{assumption}[theorem]{Assumption}

\theoremstyle{plain}

\newtheorem*{thm-non}{Theorem}

\theoremstyle{definition}

\newtheorem{defn}[theorem]{Definition}
\newtheorem{remark}[theorem]{Remark}

\definecolor{WowColor}{rgb}{.75,0,.75}
\definecolor{SubtleColor}{rgb}{0.9,0,0}

\newcounter{margincounter}

\newcounter{latercounter}

\begin{document}
\allowdisplaybreaks

\title[Convergence and Boundary Decay Rates]{Convergence Rates of Ordering, Testing and Estimation Procedures for Graphons With Fast Boundary Decay Rates
}

\author{Jeannette Janssen, Na Lin and Aaron Smith}







\begin{abstract}

In latent-position random graph models (LPMs), latent vertex positions $U_{1},\ldots,U_{n}$ are sampled from some distribution on a latent space $\Omega$, then edges of an observed graph $G = ([n],E)$ are sampled with some probability $\mathbb{P}[(i,j) \in E ]=w(U_i,U_j)$ that depends on the unobserved latent positions. LPMs are ubiquitous in the statistical analysis of networks, offering models that have good empirical performance, strong theoretical guarantees, and tractable algorithms. The special case $\Omega = [0,1]$ is important, as it corresponds to graphs with temporal or preference-based structure. In this paper, we study three problems related to LPMs with latent space $[0,1]$: \textit{ordering} the vertices according to the latent positions, \textit{estimating} the generating graphon $w$, and \textit{testing} whether an observed graph $G$ could have come from an LPM with state space $[0,1]$.  Our results on the ordering problem greatly generalize two observations of \cite{janssen2022reconstruction}: (i) for \textit{some} families of graphons, the best estimate of the ordering converges much faster than the usual statistical rate of $\frac{1}{\sqrt{n}}$, and (ii) this occurs even though, for the same families of graphons, the best estimate of the latent positions still occurs at the usual $\frac{1}{\sqrt{n}}$ rate. As a main consequence, we develop a computationally-efficient graphon-estimation algorithm and show that it has the same convergence rate as the non-explicit optimal algorithm of \cite{gao_rog}. We also derive and analyze a testing procedure.
\end{abstract}

\maketitle

\section{Introduction}

In this paper, we propose and analyze several algorithms for doing statistical inference on LPMs. Our motivating problem is the \textit{seriation} problem, first introduced in \cite{petrie1899} in the context of archaeology (see also more recent surveys \cite{laur17,liiv10}). A stylized version of the problem in \cite{petrie1899} is: if I observe several styles of urn across several grave sites, how can I order the styles  from oldest to most recent? In order to make this question tractable, we must make some assumptions; the usual assumption in  seriation is that two urn styles are more likely to be found at the same site if they were popular at similar times. 

This concrete problem gives rise to a more abstract ordering problem: given a graph $G = (V,E)$ (with vertices corresponding to urn styles, and edges between urn styles that occur in the same site), how should I order the vertices of $G$? In order to make this more abstract problem tractable, we must make an analogous assumption. We will state the assumption in the language of Robinson graphons, which we now introduce. Recall that a \textit{graphon} is a symmetric measurable function $w \, : \, [0,1]^{2} \mapsto [0,1]$ (see \textit{e.g.} \cite{lovasz2012large} for a survey). Under suitable regularity conditions, a graphon defines an algorithm for sampling a random graph of any size $n \in \mathbb{N}$: 

\begin{enumerate}
    \item Sample i.i.d. sequences $\{U_{i}\}_{i=1}^{n}, \{ U_{i,j} \}_{1 \leq i < j \leq n} \stackrel{i.i.d.}{\sim} \mathrm{Unif}([0,1]),$ then
    \item Let $V = [n] \equiv \{1,2,\ldots, n\}$ and define the edge set by 
    \be \label{eq:graphon_sampling}
    \{(i,j) \in E \} \Leftrightarrow \{U_{i,j} < w(U_{i},U_{j}) \}.
    \ee 
\end{enumerate}

We write $G \sim w$ if $G$ is a random graph obtained from graphon $w$ in this way, and call the sequence $\{U_{i}\}_{i=1}^{n}$ the \textit{latent positions} of the vertices. Following \cite{chuangpishit2015linear}, we say that a graphon $w$ is \textit{Robinson} or \textit{diagonally increasing} if it satisfies:
\be \label{EqDefGraphonRobinson}
w(x,y) \leq \min \{ w(x,z), w(z,y) \} 
\ee 
for all triples $0 \leq x < z < y \leq 1$. 

If we interpret the latent positions of the vertices as the times that they occur, we can see from Equation \eqref{eq:graphon_sampling} that the Robinson property  \eqref{EqDefGraphonRobinson} is exactly the assumption that two styles are more likely to occur in the same site if they were popular at similar times. 

In this more general setting, the seriation problem becomes: given an observed graph $G = ([n],E) \sim w$, find the permutation $\sigma_{true}\in S_n$ induced by the latent positions. We use the convention
\[
\sigma_{true}(i)=|\{j\in[n]:U_j\leq U_i\}|,
\]
so that $\sigma_{true}(i)<\sigma_{true}(j)$ if and only if $U_i<U_j$. It is natural to try to solve this problem via a two-step procedure along the following lines. For any embedding $\phi:[n]\to\mathbb{R}$, define the induced permutation $\sigma\in S_n$, breaking ties by original index, by:
\be \label{EqEstSigma}
\sigma(i)=\sigma(i,\phi)\equiv 1+ |\{j\in[n]:\phi(j)< \phi(i)\text{ or }(\phi(j)=\phi(i)\text{ and }j< i)\}|.
\ee 
We can then estimate $\sigma_{true}$ by first finding an estimate of the latent positions $\hat U_1,\ldots,\hat U_n$, then computing $\hat\sigma=\sigma(\cdot,\hat U)$.

\begin{remark}
The embedding $\phi=(U_1,\ldots,U_n)$ and its reverse $1-\phi$ are equally valid line embeddings, and the corresponding permutations cannot be distinguished from the observed graph. When comparing a permutation $\sigma$ with $\sigma_{true}$, we compare it with whichever of $\sigma_{true}$ and its reverse aligns better with $\sigma$.

\end{remark}

It is well-known that, in the special ``zero noise" case that $w$ is $\{0,1\}$-valued and no two rows of the observed adjacency matrix are identical, one can obtain exactly the correct permutation $\sigma_{true}$ by estimating the latent positions with a spectral embedding and then applying Equation \eqref{EqEstSigma} (see \textit{e.g.} \cite{atkins1998spectral}). In our previous work \cite{janssen2022reconstruction}, we showed the perhaps-surprising fact that this two-stage workflow can be far from optimal in the following sense: for many natural classes of graphons and appropriate notions of error, (i) the best embedding $\hat{\phi}$ has ``error" at least $\Omega(\frac{1}{\sqrt{n}})$, (ii) the permutation $\sigma(\cdot, \hat{\phi})$ obtained from an estimate $\hat{\phi}$ can have  ``error" at least as large as the error in the estimate, but despite this (iii) the optimal estimator of the permutation has ``error" going to 0 at the much faster rate $O(n^{-1+\epsilon})$ for any fixed $\epsilon > 0$. 

\subsection{Main Contributions}

{The main contributions of this paper} are (i) a substantial generalization of the ``faster-than-$\frac{1}{\sqrt{n}}$-convergence" result from \cite{janssen2022reconstruction} and (ii) applications of our new result to improve algorithms on the closely-related problems of graphon estimation and testing. We note that our results appear at first glance to contradict the minimax optimality results of \textit{e.g.} \cite{giraud2023mini}. See Section \ref{SecRelWork} for a short discussion of how to resolve this apparent contradiction.

Our results apply to all graphons that satisfy a set of broad assumptions, stated below as Assumptions \ref{DecayRateAssume} and \ref{U-r-relationship}. 
To help intuition, we introduce the following family of graphons indexed by noise rate $0<p\leq 1$, decay rate $0 \leq \alpha < 1$, and radius $0<r<0.5$: 
\be \label{ExSimpleGraphon}
w(x,y) &= \frac{p}{ r^{\alpha}}(r-d(x,y))^{\alpha} , \qquad &d(x,y) \leq r\\
w(x,y) &= 0, \qquad  &d(x,y) > r,
\ee 
where we think of $\alpha = 0$ as corresponding to the graphon:
\be \label{ExSimpleGraphon2}
w(x,y) &= p, \qquad &d(x,y) \leq r \\
w(x,y) &= 0, \qquad &d(x,y) > r.
\ee
See Figure \ref{fig:graphon-comparison} for a sketch representing the two types.  It is straightforward to verify that this family of graphons satisfies Assumption \ref{DecayRateAssume} with $\rho=r$ and $M_1=1$, and Assumption \ref{U-r-relationship} with $B=1$.

\begin{figure}[H]
    \centering
    \includegraphics[width=0.45\linewidth]{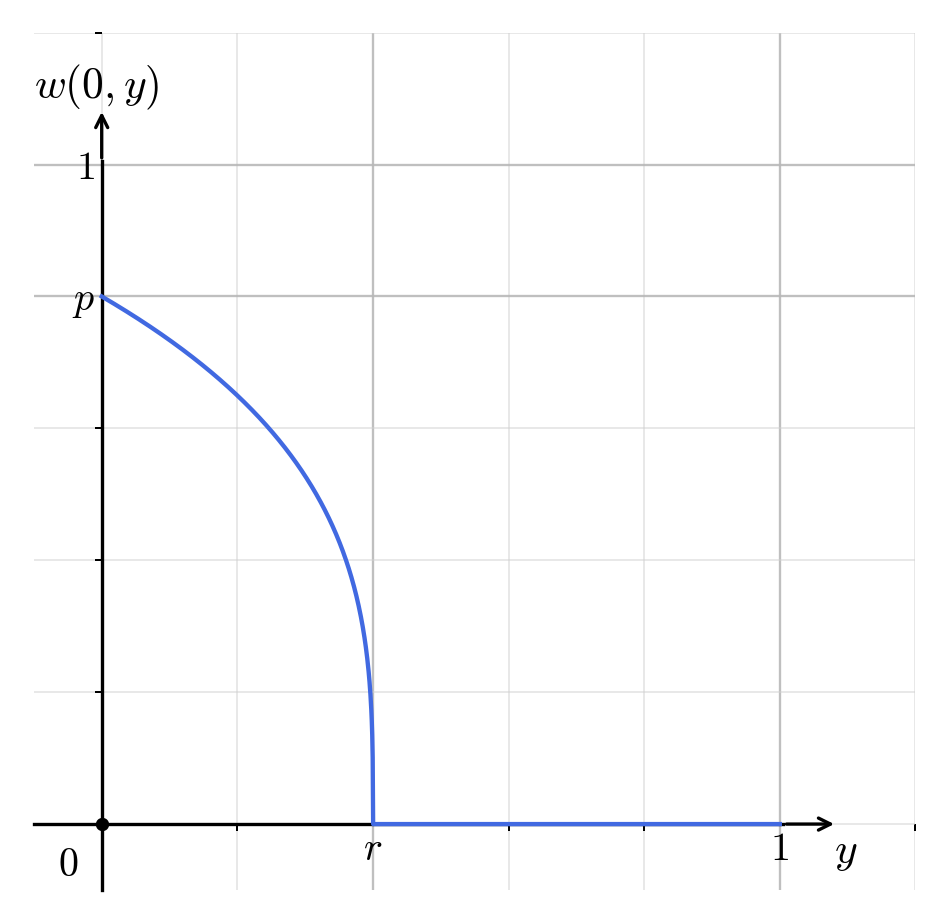}
    \hfill
    \includegraphics[width=0.45\linewidth]{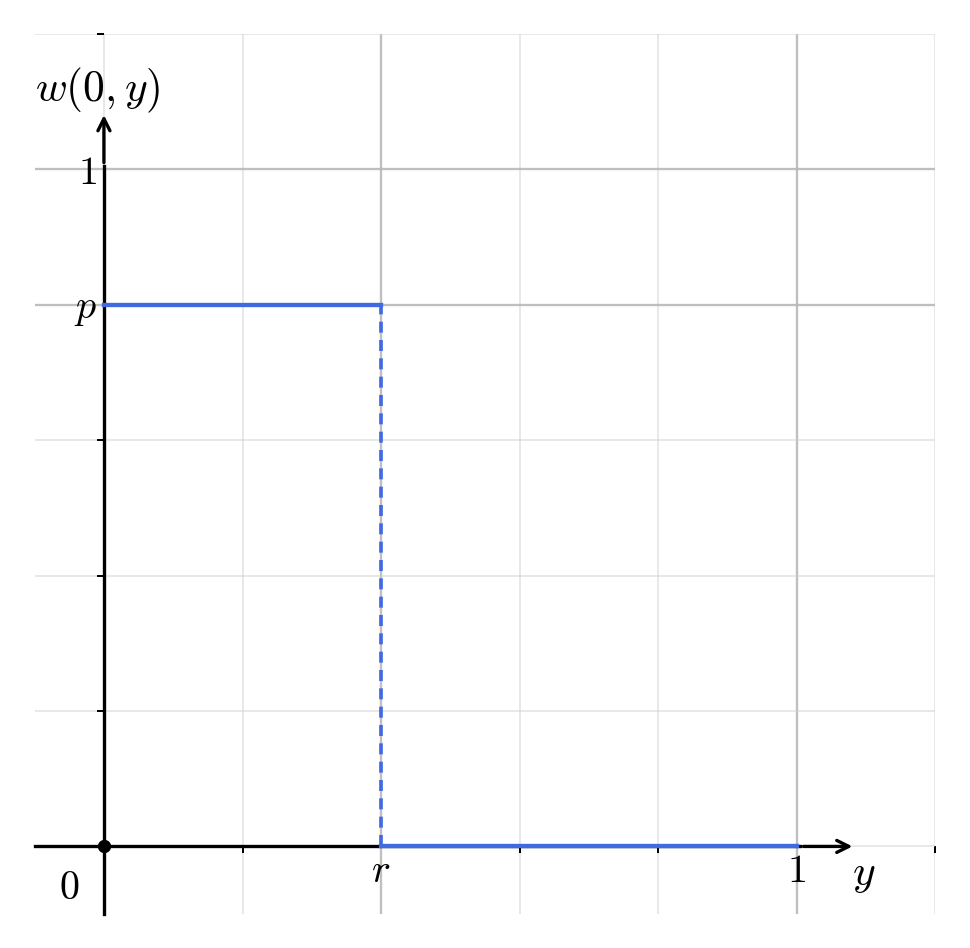}
    \caption{Sketches of $w(0,\cdot)$ for two graphon types. The left panel corresponds to \eqref{ExSimpleGraphon} with $\alpha=0.3$, while the right panel corresponds to the special case \eqref{ExSimpleGraphon2} with $\alpha=0$.}
    \label{fig:graphon-comparison}
\end{figure}

We emphasize that our results apply to a much broader and nonparametric class of graphons - these are simply intended as illustrations. 

Before stating the results, we first define the error measure between an estimated ordering and the true ordering.

\begin{defn}[Ordering Error]\label{def:ordering_error}
Let $\{U_{i}\}_{i=1}^{n}$ be a collection of distinct points in $[0, 1]$. Let $\sigma_{\text{true}}$ be the permutation induced by $(U_{1}, \ldots, U_{n})$ and let $\sigma_{rev}$ be its reverse. Given a permutation $\sigma$, the \emph{error} in $\sigma$ with respect to $\sigma_{true}$ is defined as:
\begin{equation*}
\mathfrak{D}=\min\{ \max_{i \in [n]} |\sigma(i)-\sigma_{true}(i)|,\max_{i \in [n]} |\sigma(i)-\sigma_{rev}(i)|\}.\nonumber
\end{equation*}
We will assume \emph{wlog} that this minimum is always achieved by $\sigma_{true}$. We will informally use the term \emph{error rate} to denote the normalized error $\mathfrak{D}/n$.
\end{defn}

\subsection*{Seriation}

Our first result is an algorithm, Algorithm \ref{alg:eras}, which improves upon any given coarse ordering to obtain an ordering that approximates $\sigma_{true}$ with an error that depends on the parameter $\alpha$. 
We view this result as a generalization of Theorem 2 of \cite{janssen2022reconstruction} from the special case $\alpha = 0$ to general $0 \leq \alpha < 1$. The precise result is stated in Section 3 as Theorem \ref{theorem-eras}; here we state an informal version of this theorem. 

\begin{theorem}
Suppose that $G$ is a graph of size $n$ sampled from a graphon that satisfies Assumptions \ref{DecayRateAssume} and \ref{U-r-relationship}, with decay rate $\alpha$. Suppose Algorithm \ref{alg:eras} is executed with input $G$ and suitable parameters, including a fixed $\epsilon >0$ small enough.  Then with high probability the resulting estimated ordering $\hat{\sigma}$ has normalized error $\mathfrak{D}/n$ satisfying
\be \label{IneqMainResInf}
\frac{\mathfrak{D}}{n} =O\left( n^{-\frac{1}{1 + \alpha} + \epsilon}\right).\nonumber
\ee  
\end{theorem}

This improves upon the usual ``statistical" rate of $\frac{1}{\sqrt{n}}$ for all $0 < \alpha < 1$, matching that statistical rate at $\alpha = 1$ and matching the rate in \cite{janssen2022reconstruction} when $\alpha = 0$. 

\subsection*{Graphon Estimation}



 In the general \textit{graphon estimation} problem, we consider a graph $G$ that is sampled from a graphon $w$, then find an estimate $\tilde{w}$ of the full graphon $w$. For samples of a Robinson graphon, one can ``easily" use an estimated vertex ordering $\sigma$ to define a ``reasonable" estimator. For example, fixing a bandwidth parameter $b \in \mathbb{N}$, define the ``local averaging" estimator:

\be \label{EqWHatNaive}
\tilde{w}_{naive}(x,y) = \frac{1}{(2b+1)^{2}} \sum_{i\in B(x)} \sum_{j\in B(y)} \textbf{1}_{(i,j) \in E},
\ee
where $B(x)=\{k: \lfloor nx\rfloor-b\leq \sigma(k) \leq \lfloor nx\rfloor+b\}$. 

Algorithm~\ref{alg:graphon-estimator} first uses Algorithm \ref{alg:eras} to obtain an estimated ordering, and then uses an estimator similar to Equation \eqref{EqWHatNaive} to estimate the graphon at the latent positions. 
The theorem below is an informal version of Theorem \ref{ourTheorem} in Section 4. 

\begin{theorem}


Suppose that $G$ is a graph of size $n$ sampled from a graphon that satisfies Assumptions \ref{DecayRateAssume}, \ref{U-r-relationship} and \ref{assumption:Holder}.  Suppose that Algorithm \ref{alg:graphon-estimator} is executed with input $G$ and suitable parameters, including a fixed $\epsilon > 0$ small
enough. Then with high probability the resulting estimator $\tilde{w}$ satisfies
\[
\frac{1}{n^2}\sum_{i,j\in [n]} (\tilde{w}(U_i, U_j)  - w(U_i,U_j))^2 = O(n^{-\frac{2\alpha}{1+\alpha}+\epsilon}).
\]
\end{theorem}

Algorithm~\ref{alg:graphon-estimator} is closely related to the naive local-averaging estimator in \eqref{EqWHatNaive}, and Theorem~\ref{ourTheorem} gives an error rate of order $O(n^{-\frac{2 \alpha}{1+\alpha}})$ up to lower-order factors. This error is known to be essentially optimal, as upon suitable rescaling it matches the minimax bounds established in \cite{gao_rog}. While the optimal rate for the graphon estimation problem was established in \cite{gao_rog}, that paper did not give a computationally-tractable algorithm. In contrast, Algorithm~\ref{alg:graphon-estimator} gives an explicit and quick way to compute an estimator that has essentially-optimal convergence rates.

\subsection*{Test Statistic}

Our final result is a computationally tractable test for whether an observed graph was sampled from a ``nice" Robinson graphon. The test uses a computable variant of the $\Gamma$ and $\Lambda$ statistics from \cite{chuangpishit2015linear} and \cite{LambdaDef24}. Theorem~\ref{Thm_Test_Power} controls this statistic under the null hypothesis; informally:
\begin{theorem}
Let $G$ be a graph of size $n$ sampled from a graphon that satisfies Assumptions~\ref{DecayRateAssume} and~\ref{U-r-relationship}.
Let $\hat{\Lambda}(G)$ be calculated as in Algorithm~\ref{alg:graphon-test} with input $G$ and appropriate parameters, including size scale $\mu=n^{-\beta}$, where $0<\beta<1$, and a sufficiently small fixed $\epsilon >0$. Then with high probability,
\begin{align*}
\hat{\Lambda}(G) & = O (n^{-1-\beta}\log(n)^2) + O\left(\min\left\{n^{-2\beta},\, n^{-\beta-\frac{1}{1+\alpha}+\epsilon}\log(n)\right\}\right).
\end{align*}
If Assumption~\ref{assumption:Holder} also holds, then the second term can be replaced by
\[
O\left( n^{-2\beta-\frac{\alpha}{1+\alpha}+\epsilon}\log(n)^\alpha\right).
\]
\end{theorem}
The first term is, roughly, the sampling fluctuation, and the second is, roughly, the ordering error.

\subsection{Related Work} \label{SecRelWork}

The problem of extracting orderings or rankings from pairwise (dis-)similarities appears in many contexts in the machine learning and statistics literatures. In some cases, the details of the underlying data-generating process are quite similar to those in the current paper; in most, they are not. Nonetheless, we view many such papers as working on essentially the ``same" problem, and we expect to be able to move algorithms and ideas between contexts.

Besides \cite{janssen2022reconstruction}, the most-similar papers on seriation that we are aware of are \cite{rocha2018recovering,fogel2014serialrank,natik2021consistencyspectralseriation,cai2023matrix,Flammarion2016OptimalRO,giraud2023mini,berenfeld2024seriationtoeplitzlatentposition,issartel2025minimaxoptimalseriationpolynomial}. While these papers are not all written in the language of ``Robinson graphons," all of them study distributions on similarity matrices that could be rewritten in the form of  Equation \eqref{eq:graphon_sampling} for some graphon satisfying Inequalities \eqref{EqDefGraphonRobinson} (though some study only a restricted class of such random similarity matrices, while others study much broader classes of similarity matrices that are not necessarily $\{0,1\}$-valued). The largest difference between our work and these works comes from the final advertised rates: our main results (and those in \cite{janssen2022reconstruction}) give rates that are faster than the statistical rate of $\frac{1}{\sqrt{n}}$, while all other papers give rates that are similar to or slower than that rate.

We now describe the papers in somewhat more detail. The first papers, \cite{rocha2018recovering,fogel2014serialrank,natik2021consistencyspectralseriation,cai2023matrix}, study essentially the same ``ordering" problem as we study. The first two consider only special cases corresponding roughly to the one-parameter graphon $w(x,y) = p \, \mathbf{1}_{|x-y|\leq 0.5}$ (where $p \in [0,1]$) and the single graphon $w(x,y) = 1- |x-y|$, while the latter two cover nonparametric families of graphons that are similar to those studied in the present paper. None of these papers prove optimality of their results (though one can easily verify that the main bound in \cite{janssen2022reconstruction} is optimal up to polylog factors).

The latter papers, \cite{Flammarion2016OptimalRO,giraud2023mini,berenfeld2024seriationtoeplitzlatentposition,issartel2025minimaxoptimalseriationpolynomial}, are both more general and prove stronger results. Between them they cover a large nonparametric family of graphons (and some seriation problems that do not correspond to graphons at all), and they show that the optimal rate of convergence for an ordering-like problem under their assumptions is $O( \sqrt{\log(n)/n} )$.  We highlight one result that might be surprising: the main results  of \cite{giraud2023mini,issartel2025minimaxoptimalseriationpolynomial} say that the \textit{optimal} rate of convergence for the ordering problem is roughly $n^{-0.5}$ under certain conditions, while our Theorem \ref{theorem-eras} says that a much faster rate of $n^{-c}$ for $0.5 < c < 1$ is possible. This is not a contradiction. The main results of \cite{giraud2023mini,issartel2025minimaxoptimalseriationpolynomial} cover a very broad class of graphons, and their lower bound only says that \textit{some} problems in that class are difficult. Our results divide this large class of graphons into many much smaller classes, according to smoothness conditions similar to those in \cite{gao_rog}, and so we are able to find classes of graphons for which convergence is much faster.

Many statisticians have studied re-ordering problems that are similar in spirit to ours but not formally equivalent. Some prominent examples include the recovery of ordering for matrices with the Monge property  \cite{hutter20} (rather than the Robinson property, as in this paper), or the recovery of ordering based on noisy pairwise estimates of which entry is higher in the ordering \cite{chaterjee15} (rather than noisy pairwise estimates of similarity, as in this paper). Broadly speaking, we expect many techniques to transfer between these problems, though it might not be possible to give formal bounds relating solutions to the different problems.

In Section \ref{SecGraphEst} of this paper, we use the output of our main seriation algorithm as input to ``downstream" algorithms for graphon estimation and hypothesis testing. Seriation is also used as an initial step in various other algorithms, such as community detection \cite{Zendehboodi_2023}.

There are large literatures on estimating graphons and doing hypothesis tests on graphons. For representative strong results in those areas, we refer the reader to \cite{gao_rog} for graphon estimation and \cite{testing08} for property testing. There are also literatures on these subjects for very specific families of graphons. As representative examples, \cite{klus2025learninggraphonsdatarandom} studies graphon estimation for random walk data obtained from a stochastic block model graphon, while \cite{LambdaDef24} studies hypothesis testing for Robinson graphons. The most similar papers to ours are \cite{gao_rog} for graphon estimation and \cite{LambdaDef24} for hypothesis testing. In both areas, our current work has two main differences to previous work:

\begin{enumerate}
    \item Our algorithm is explicit and computationally tractable, and
    \item Our computational advantages rely very strongly on our good estimates of the underlying vertex ordering.
\end{enumerate}

In the case of graphon estimation, our guarantees essentially match the guarantees of \cite{gao_rog} upon appropriate rescaling. In the case of property testing, we improve upon the rates of \cite{LambdaDef24} at the cost of making stronger assumptions.

\subsection{Guide to Paper}
We give notation and state our main assumptions in Section \ref{SecNotation}, describe and analyze our main algorithm for vertex ordering in Section \ref{section-error-rooting}, and describe and analyze applications to graphon estimation and hypothesis testing in Section \ref{SecApplications}.

\section{Notation} \label{SecNotation}

\subsection{General Notation}

We now define a notion of error in the ordering that depends directly on the positions $U_i$ rather than on the permutation derived from these values and known as $\sigma_{true}$.  

\begin{defn}\label{Def-precision-level}
Given $S \subseteq  V$, a permutation $\sigma$ of $S$, and a permutation $\tau$ of $V$, we say that 
$\sigma$ \emph{agrees with} $\tau$ to \emph{precision level} $d$, if, for all $i,j \in S$ such that 
$|U_i - U_j| \ge d$, we have $\sigma(i) < \sigma(j)$ if and only if $\tau(i) < \tau(j)$.
If $\sigma$ agrees with $\sigma_{true}$ (or $\sigma_{rev}$) to precision level $d$, we will say that $\sigma$ has precision level $d$.
\end{defn}

We say that a sequence of events $\mathcal{A} = \{\mathcal{A}^{(n)}\}$ indexed by $n \in \mathbb{N}$ hold \emph{with extreme probability} if
\[
\mathbb{P}\left[\mathcal{A}^{(n)}\right] \ge 1 - n^{-c_1 \log(n)^{c_2}}
\]
for some $c_1, c_2 > 0$ and all $n > N_0$ sufficiently large. We use the shorthand $\mathcal{A}$ holds w.e.p.

When $f,g$ are two nonnegative functions with the same domain, we write $f = O(g)$ as the shorthand for: there exists a positive constant $C<\infty$ such that $ f(x) \leq C g(x)$ for all $x$. Similarly, we write $g = \Omega(f)$ if $f = O(g)$, and we write $f=\Theta(g)$ if both $f = O(g)$ and $g = O(f)$. 

\subsection{Main Assumptions}

For this section, fix a Robinson graphon $w$.

\begin{defn} [Domain]
Define $r(x) = \sup\{y \, : \, w(x,y) > 0\}$ and $\ell(x) = \inf\{ y \, : \, w(x,y) > 0\}$.
\end{defn}

\begin{assumption} [Decay Rate]\label{DecayRateAssume}
There exist constants $0<\alpha<1$, $0< \rho  <0.5$ and $1\leq M_{0}, \, M_1<\infty$ such that: 
\begin{enumerate}
\item For all $x \in [0,1- \rho]$, we have $r(x)-x\geq \rho$ and for $z \in [0, r(x)-x]$
    \be
     \frac{w(x,r(x) - z)}{z^{\alpha}} \in [M_{0}^{-1},M_{0}].\nonumber 
    \ee 
\item For all $x \in [ \rho ,1]$, we have $x-\ell(x)\geq  \rho $ and for $z \in [0,x-\ell(x)]$
    \be
    \frac{w(x,\ell(x) + z)}{z^{\alpha}} \in [M_{0}^{-1},M_{0}].\nonumber 
    \ee 
\item Fix $0<d<1- \rho $. For all $x,y \in [0,1- \rho]$ with $x-y=d$ and for $z\in [0,r(x)-x]$,  
    \be
    \frac{w(x,r(x) - z)- w(y,r(x)-z)}{z^{\alpha}- \max\{0,(z-d)^{\alpha}\}}\in [M_1^{-1}, M_1].\nonumber 
    \ee 
\item Fix $0<d<1- \rho$. For all $x,y \in [ \rho,1]$ with $x-y=d$ and for $z \in [0, y-\ell(y)] $
    \be
    \frac{w(y,\ell(y)+z)- w(x,\ell(y) + z)}{z^{\alpha}- \max\{0,(z-d)^{\alpha}\}}\in [M_1^{-1}, M_1].\nonumber 
    \ee 
\end{enumerate}
\end{assumption}

In informal discussions, we refer to the constant $\alpha$ of a graphon satisfying Assumption \ref{DecayRateAssume} as the \textit{decay rate}.

\begin{assumption}[Propagation of Differences to Boundaries] \label{U-r-relationship}
There exists some $B > 0$ so that for all $0 < x < y < 1$,
\[
(r(y)-r(x))+(\ell(y)-\ell(x))  \geq B|x - y|.
\]
\end{assumption}

\section{Ordering Algorithm}\label{section-error-rooting}

We present and analyze our main ordering algorithm, Algorithm \ref{alg:eras}. 
The algorithm is presented in Section \ref{SubsecOrdAlg}, a single round is analyzed in Section \ref{SubsecOrdOnce}, and the full algorithm is analyzed in Section \ref{SubsecOrdMany}. Throughout this section, we fix a graphon $w$ that satisfies Assumptions \ref{DecayRateAssume} and \ref{U-r-relationship} and consider graphs of size $n$ sampled from that graphon.

As in \cite{janssen2022reconstruction}, the main idea is to fix a growing sequence of sets of vertices $V_{1} \subsetneq V_{2} \subsetneq \ldots \subsetneq V_{k} = V$ and, at each round $j \geq 1$, use an approximate ordering of $V_{j}$ to get a refined ordering of $V_{j+1}$. Critically, the estimate of the ordering of $V_{j+1}$ depends only on (i) edges between pairs of vertices in $V_{j}$ and (ii) edges between $V_{j}$ and $V_{j+1} \backslash V_{j}$ (that is, it ignores edges between pairs of vertices in $V_{j+1} \backslash V_{j}$ that are being embedded in the same round). This allows us to ensure that the random variables used in different stages of the algorithm are independent, simplifying the analysis.

To generate the nested sequence of vertex sets, we introduce i.i.d. random variables $\{B_i\}_{i=1}^n \stackrel{i.i.d.}{\sim} \mathrm{Unif}([0,1])$. For a given sequence of parameters $0<p_1<p_2<\cdots<p_k=1$ (to be introduced later), we define 
\[
V_i=\{j\in V: B_j \leq p_i\}, \quad 1\leq i \leq k.
\]
This randomization provides a systematic way to determine the sequence in which vertices are progressively included.

\subsection{Main Ordering Algorithm} \label{SubsecOrdAlg}

We begin with a heuristic explanation for our main building block, Algorithm \ref{alg:erss}, and introduce the notation necessary for the algorithm as we go.

Our main heuristic, based on Assumption \ref{DecayRateAssume}, is as follows: if $U_i<U_j$, then there is a region $R\subset [0,1]$ such that, for each vertex $k$ with $U_{k} \in R$, the probability $w(U_{i},U_{k})$ that $i$ is a neighbor of $k$ is much smaller than the probability $w(U_{j},U_{k})$ that $j$ is a neighbor of $k$. If $U_i$ and $U_j$ are sufficiently far apart, then $R$ is large and will contain many vertices; these will provide the signal that indicates the true ordering of $i$ and $j$.
Figure~\ref{fig:R-Region} provides a toy illustration of this heuristic for the family of graphons in \eqref{ExSimpleGraphon}.
Although we do not know exactly which vertices lie in $R$, we know that the vertices that provide the signal must be at the extremes of the true ordering in the neighborhood of $i$ or $j$ - they must be close to $r(U_{i}), r(U_{j}), \ell(U_{i}),$ or $\ell(U_{j})$. The ordering from the previous round can therefore be used to estimate which vertices are in $R$. 

\begin{figure}[h]
    \centering
    \includegraphics[width=0.7\textwidth]{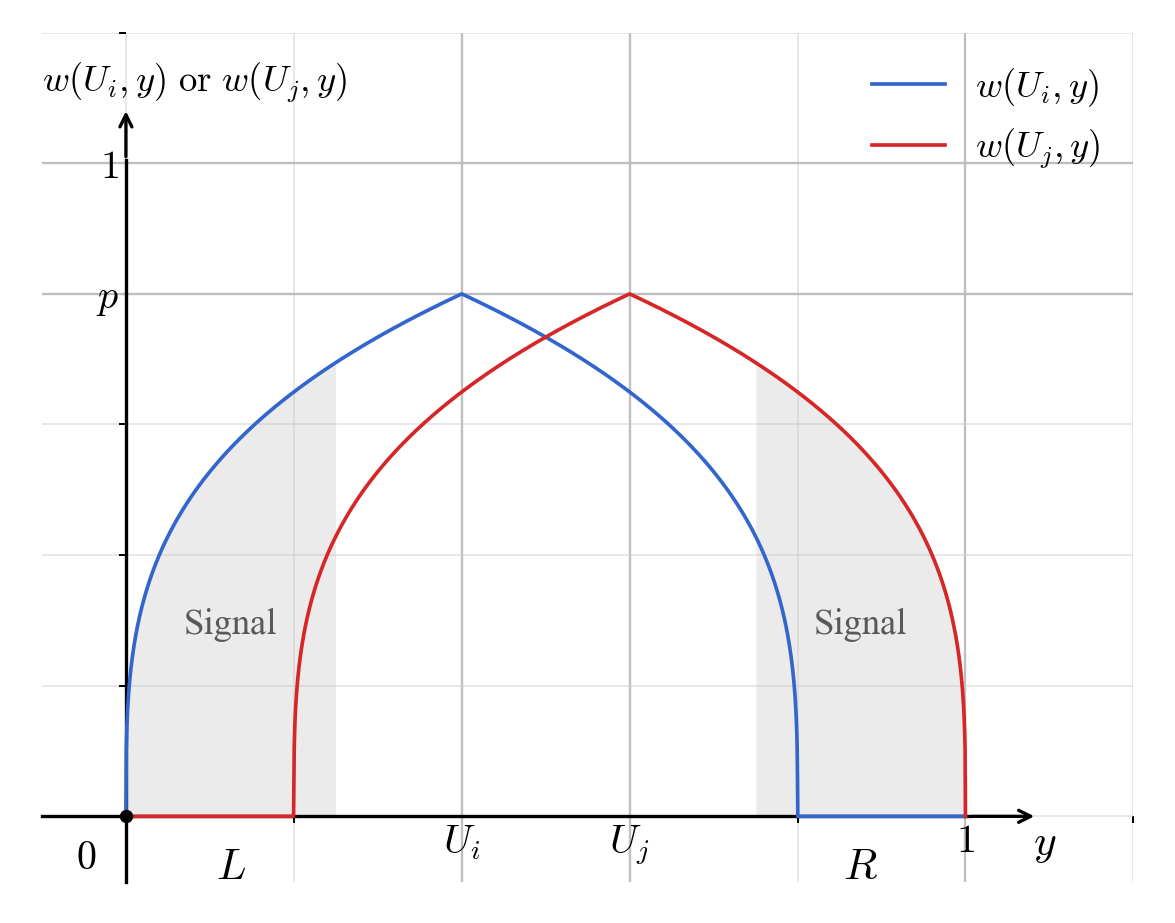}
    \caption{Illustration of ordering signals from extreme neighborhoods for the family of graphons in \eqref{ExSimpleGraphon}.}
    \label{fig:R-Region}
\end{figure}

To make this precise and give our algorithm, we introduce some definitions. Given an ordering $\sigma$ of a set $V$, a set $S \subset V$ , and a parameter $c < |S|$, we define the sets of the $c$ elements of $S$ with the highest and lowest rank according to $\sigma$:
\be \label{EqRLDef}
R(S, \sigma, c) &= \{k : |\{p \in S : \sigma(k) \leq \sigma(p)\}| \leq c\}, \\ 
L(S, \sigma, c) &= \{k : |\{p \in S : \sigma(p) \leq \sigma(k)\}| \leq c\}.
\ee
Given a permutation $\gamma$ on $S$, define
\[
F_{\gamma}(i,j)=\left\{
\begin{array}{ll}
1 & \mbox{if $\gamma(i)<\gamma(j)$,}\\
0 & \mbox{if $\gamma(i)=\gamma(j)$,}\\
-1 & \mbox{if $\gamma(i)>\gamma(j)$.}
\end{array}
\right.
\]
For a comparison function $F:S^2\to\{-1,0,1\}$, set
\[
\gamma_F(i)=\sum_{j\in S} F(i,j),
\]
\[
\sigma_F(i)=1+|\{j\in S: \gamma_F(j)<\gamma_F(i)\mbox{ or ($\gamma_F(j)=\gamma_F(i)$ and $j<i$)}\}|.
\]
Thus larger values of $\gamma_F$ are ranked later.

\begin{remark}
    In our algorithm, $F$ need not come from a permutation. For example, $F(i,j)=0$ may hold even when $i\not= j$. The tie-breaker $j< i$ ensures that $\sigma_F$ is still a permutation.
\end{remark}

With this notation, we can present our ``single round" algorithm as Algorithm \ref{alg:erss} below. We define $SingleStage(G_2,V_1,\sigma_1)$ to be the output produced by Algorithm \ref{alg:erss} given the graph $G_2=G[V_2]$ (corresponding to $V_2$), $V_1$ and $\sigma_1$.

\begin{algorithm}[h]
\caption{Error-Rooting: Single Stage}\label{alg:erss}
\begin{algorithmic}[1]
\algrenewcommand\alglinenumber[1]{} 
\State \textbf{parameter:} $C_1$, $C_2$, $C_3$.
\State \textbf{input:} Observed graph $G=(V,E)$, sets $V_1 \subset V_2 \subset V$, an order $\sigma_1$ on $V_1$.
\State \textbf{output:} An order $\sigma_2$ on $V_2$.

\algrenewcommand\alglinenumber[1]{\scriptsize #1}
\setcounter{ALG@line}{0}
\State Initialize $F(i,j)=0$ for all ordered pairs $(i,j)\in V_2^2$.
\For{$i<j$ in $V_2 \setminus V_1$}
\State Set
\[
R=R(i,j)=R(V_1 \cap (N(i)\cup N(j)),\sigma_1,C_1)
\]
\[
L=L(i,j)=L(V_1 \cap (N(i)\cup N(j)),\sigma_1,C_1)
\]
according to Equation \eqref{EqRLDef}.
\If{$|N(j)\cap R|-|N(i)\cap R|>C_2$ or $|N(i)\cap L|-|N(j)\cap L|>C_2$}
\State Set $F(i,j)=1$ and $F(j,i)=-1$.
\ElsIf{$|N(i)\cap R|-|N(j)\cap R|>C_2$ or $|N(j)\cap L|-|N(i)\cap L|>C_2$}
\State Set $F(i,j)=-1$ and $F(j,i)=1$.
\EndIf
\EndFor
\State Set $\sigma'_2 =\sigma_{F}$ on $V_2\setminus V_1$, using the restriction of $F$ to $(V_2\setminus V_1)^2$.
\For{$i$ in $V_2$}
\State  Set
\[
t(i)=\max\{\sigma'_2(k): k \in  N(i)\cap (V_2 \setminus V_1)\},
\]
\[
b(i)=\min\{\sigma'_2(k): k \in  N(i)\cap (V_2 \setminus V_1)\}.
\]
\EndFor
\For{$i<j$ in $V_2$ with $i\in V_1$ or $j\in V_1$}
\If{$t(j)-t(i)>C_3$ or $b(j)-b(i)>C_3$}
\State Set $F(i,j)=1$ and $F(j,i)=-1$.
\ElsIf{$t(i)-t(j)>C_3$ or $b(i)-b(j)>C_3$}
\State Set $F(i,j)=-1$ and $F(j,i)=1$.
\EndIf
\EndFor
\State Return $\sigma_2=\sigma_{F}$ on $V_2$.
\end{algorithmic}
\end{algorithm}

To achieve high accuracy, we iteratively apply this algorithm, choosing the sizes of the growing sequence of subsets carefully. This does require a separate initialization algorithm that  takes as input a graph and outputs a (possibly quite inaccurate) initial ordering. In Algorithm \ref{alg:eras} we just describe this as a ``coarse algorithm", and we assume this algorithm returns an ordering with error that converges at rate $n^{-\gamma}$, where $\gamma>0$ is called the \emph{precision parameter}. We will see in the statement of Theorem \ref{theorem-eras} that we don't require this coarse algorithm to be very accurate - convergence at rate $n^{-\gamma}$ for \textit{any} $\gamma > 0$ is sufficient, even if $\gamma \ll 0.5$ is much worse than the statistical rate. In particular, the coarse algorithm in \cite{janssen2022reconstruction} is sufficient for our purposes (when its assumptions hold), and we suspect essentially any of the algorithms described in Section \ref{SecRelWork} would also be sufficient.

The complete multi-stage procedure is detailed in Algorithm \ref{alg:eras}:

\begin{algorithm}[h]
\caption{Error-Rooting: All Stages}\label{alg:eras}
\begin{algorithmic}[1]
\algrenewcommand\alglinenumber[1]{} 
\State \textbf{parameters:} Number of rounds $k$, sizes $0<p_1<p_2<\cdots<p_k=1$, target errors $d_1>d_2>\cdots>d_k$, decay rate $\alpha \in [0,1]$, coarse algorithm $P$ with precision parameter $\gamma$.
\State \textbf{input:} Observed graph $G=(V,E)$.
\State \textbf{output:} A permutation $\sigma$ on $V$.

\algrenewcommand\alglinenumber[1]{\scriptsize #1}
\setcounter{ALG@line}{0}
\State Let $\{B_i\}_{i=1}^n \overset{i.i.d.}{\sim} \mathrm{Unif}[0,1]$.
\State Run the coarse algorithm $P$ on the induced subgraph $G_1$ of $G$ with vertex set
$V_1 = \{j \in V : B_j \leq p_1\}$ and $n_1 = |V_1|$ to get the rough ordering $\sigma_1$ of $V_1$.
\For{$i$ in $\{1,2,\ldots, k-1\}$}
\State Let $G_{i+1}$ be the induced subgraph of $G$ with vertex set $V_{i+1} =\{j \in V : B_j \leq p_{i+1}\}$.
\State Run Algorithm \ref{alg:erss} with parameters 
$$C_1=\left\lceil np_i d_i\log(n) \right\rceil, $$
$$C_2= \left\lfloor \frac{1}{6M_1} (n p_i)^{1/2}(d_i)^{(1+\alpha)/2} \right\rfloor,$$
$$C_3=\left\lfloor \frac{1}{2(1+\alpha)}(n p_i)^{(1-\alpha)/2}  (d_i)^{(1-\alpha^2)/2} (\log (n))^{1+\alpha}\right\rfloor.$$
Let $\sigma_{i+1}=SingleStage(G_{i+1}, V_i, \sigma_i)$ be the result.
\EndFor
\State Return $\sigma=\sigma_k$.
\end{algorithmic}
\end{algorithm}

\subsection{One-Step Analysis} \label{SubsecOrdOnce}

We now analyze Algorithm \ref{alg:erss}. In the next section, we will see how to choose the parameters $C_1,C_2,C_3$ so that the properties described in this section hold asymptotically almost surely.

Algorithm \ref{alg:erss} builds a new ordering $\sigma_2$ on $V_2$ by looking at the edges in $G_2$ and the old ordering $\sigma_1$ on $V_1 \subset V_2$. 
To start the heuristic analysis, we assume that $\sigma_1$ is an ordering of $V_1$ which agrees with $\sigma_{true}$ at precision level $d_1$. We now show how $\sigma_1$ is used in Algorithm \ref{alg:erss} to find an ordering $\sigma_2$ of $V_2 \backslash V_{1}$ which agrees with $\sigma_{true}$ at smaller precision level $d_2$.

Fix $i,j\in V_{2}\setminus V_{1}$ so that $|U_{i}-U_{j}|>d_{2}$. Assume without loss of generality that $U_{i}<U_{j}$. Let 
\be \label{Itrue}
I_{true}(i,j)=[r(U_{j})-d_1,r(U_{j})].
\ee
Recall that $U_j$ cannot have any neighbors greater than $r(U_j)$. Therefore, the neighbors of $U_j$ in $I_{true}$ will be those with highest value in $\sigma_{true}$. Because of the distinction between $w(U_i,x)$ and $w(U_j,x)$ presented in Assumption \ref{DecayRateAssume} and \ref{U-r-relationship}, the region $I_{true}$ is very informative in distinguishing the ordering of $i,j$. The idea behind the algorithm is that the set $R(i,j)$ (as defined in line 3 of Algorithm \ref{alg:erss}) is a small set that includes most of the ``signal'' vertices in $I_{true}$ and as few additional ``noisy'' vertices as possible; when this works, we expect the vertex with larger value $U_j$ to have significantly more neighbors in $R(i,j)$. 

We now turn this heuristic into a proof for appropriate hyperparameters $C_{1},C_{2},C_{3},p_{1},p_{2}$, with estimates deferred to Section \ref{SubsecDeferredLemmas}:

\begin{itemize}
\item \textbf{Large Signal:} Denote by
\[
Dist(i,j)=\{k\in V_1: U_k\in I_{true}(i,j)\}
\]
the set of ``signal" vertices that strongly distinguish between $i$ and $j$ on the right.
We show in Lemma \ref{DistLarge} below that w.e.p. $|Dist(i,j)|>n p_1 d_1 \log(n)^{-1}$. We define
\be \label{EqDefSignal}
Signal(i,j)= |R(i,j)\cap Dist(i,j)\cap N(j)|- |R(i,j)\cap Dist(i,j)\cap N(i)|.
\ee 
We will show in Lemma \ref{SignalBelongstoR} that w.e.p.
\[
Dist(i,j)\cap ( N(j) \cup N(i)) \subseteq R(i,j),
\]
and hence
\begin{align*}
Signal(i,j) &= |Dist(i,j)\cap N(j)|-|Dist(i,j)\cap N(i)|.
\end{align*}
Moreover, we show in Lemma \ref{LowerBoundSignal} that, for appropriately chosen $C_2$, w.e.p.
\[
Signal(i,j) \geq 2C_2.
\]

\item \textbf{Small Noise:} We define
\be \label{EqDefNoise}
Noise(i,j)=|R(i,j)\cap N(i)\setminus Dist(i,j)|-|R(i,j)\cap N(j) \setminus Dist(i,j)|.
\ee
Then it is easy to see that
\be\label{Signal-Noise}
|R(i,j)\cap N(j)|-|R(i,j)\cap N(i)|=Signal(i,j)-Noise(i,j).
\ee
We will show in Lemma \ref{UperBoundNoise} below that w.e.p., $Noise(i,j)<C_2$.
\end{itemize}

Putting together the above estimates, we will see that, in this situation, $|R(i,j)\cap N(j)|-|R(i,j)\cap N(i)| > C_{2}$ w.e.p., and so in particular $F(i,j)$ is set to the correct value in Steps 4-8 of Algorithm \ref{alg:erss}. This is enough to show that the ordering of $V_{2} \backslash V_{1}$ has high accuracy. The analysis of the extension to all of $V_{2}$ is essentially independent of the above heuristic, and is analyzed in Lemma \ref{LemmaExtCorrect}.

\subsubsection{Deferred Estimates} \label{SubsecDeferredLemmas}

We now prove the deferred estimates. In this section, we assume that $p_{1},p_{2},d_{1},d_{2},C_{1},C_{2},C_{3}$, all depending on the total graph size $n$, satisfy the following equalities and inequalities: 
\be
\label{parameter}  \log(n)^{-3} \leq \ p_1 < p_2 &\leq 1  \\
\nonumber n^{-1/(1+\alpha)}< d_1 & \leq \log(n)^{-2},\\
\nonumber d_2 &= (n p_1)^{-\frac{1}{2}}(d_1)^{\frac{1-\alpha}{2}},\\
\nonumber C_1 &= \left\lceil np_1 d_1\log(n) \right\rceil,\\
\nonumber C_2 &= \left\lfloor \frac{1}{6M_1} (n p_1)^{\frac{1}{2}}(d_1)^{\frac{1+\alpha}{2}} \right\rfloor,\\
\nonumber C_3 &= \left\lfloor \frac{1}{2(1+\alpha)}(n p_1)^{\frac{1-\alpha}{2}}  (d_1)^{\frac{1-\alpha^2}{2}} (\log (n))^{1+\alpha}\right\rfloor,
\ee
where $\alpha$, the decay rate, and $M_1$ are as in Assumption \ref{DecayRateAssume}.

We also assume that $\sigma_1$ is an ordering of $V_1$ that has precision level $d_{1}$, which was estimated using only edges internal to $V_{1}$.

We next show that, with input and parameters as given in Equation \eqref{parameter}, Algorithm \ref{alg:erss} does indeed produce an ordering $\sigma_2$ which has precision level $d_2$. Recall that $\{B_i\}_{i=1}^n \stackrel{i.i.d.}{\sim} \mathrm{Unif}([0,1])$.  For a given $p$, let $S=S(p)=\{i:B_i \leq p\}$ and define
\begin{align*}
\mathcal{A}_1(p)=\bigcup_{k,l\in V, U_l-U_k>\frac{1}{np}}\Big\{\Big|  |\{i\in S:U_k<U_i<U_l \}| & -np|U_l-U_k| \Big|\\
&>2\sqrt{np|U_l-U_k|}\log (n)\Big\},
\end{align*}
the event that some interval in the latent space has ``far too many" or ``far too few" observations. By Chernoff's inequality, for any given sequence $p=p(n)\in (0,1)$ satisfying $\lim \sup_{n\rightarrow \infty} \frac{-\log(p(n))}{\log(n)}<1$, we have that $\left(\mathcal{A}_1(p)\right)^c$ holds w.e.p. Informally, this means that all observed intervals will have ``about the right number" of observations w.e.p. We will use this fact frequently in the remainder of this section.

\begin{lemma}\label{DistLarge}
W.e.p. for all $i,j\in V_2\setminus V_1$ satisfying $U_j-U_i>d_2$,
\[
|Dist(i,j)|\geq np_1 d_1 \log(n)^{-1}.
\]
\end{lemma}
\begin{proof}
Recall the definition of $I_{true}$ from Equation \eqref{Itrue}. By definition, $|I_{true}(i,j)|=d_1$ and the set $Dist(i,j)=\{k\in V_1: U_k\in I_{true}(i,j)\}$ has expected size
\[
\mathbb{E}[|Dist(i,j)|]=np_1 d_1.
\]
We also note that $n^{-1/(1+\alpha)}<d_1 <1 $. Combining these two facts, for all $i,j\in V_2\setminus V_1$,
\[
\{|Dist(i,j)|<n p_1d_1 \log(n)^{-1}\} \subset \mathcal{A}_1(p_1).
\]
Since w.e.p. $\left(\mathcal{A}_1(p_1)\right)^c$ holds, the lemma follows.
\end{proof}

The following lemma follows directly from arguments similar to those used in the proof of Lemma 4.4 in \cite{janssen2022reconstruction}.

\begin{lemma}\label{SignalBelongstoR}
W.e.p. for all $i,j\in V_2\setminus V_1$ satisfying $U_j-U_i>d_2$,
\[
Dist(i,j)\cap (N(j) \cup N(i))\subseteq R(i,j).
\]
\end{lemma}

\begin{proof}
Suppose $ Dist(i,j)\cap (N(j) \cup N(i)) \not\subset R(i,j)$. Let
\[
k=\mbox{argmax}\{\sigma_1(s): s\in \left( Dist(i,j)\cap (N(j) \cup N(i))\right)\setminus R(i,j)\}.
\]
That is, $k$ is the most highly ranked vertex in the set $ Dist(i,j)\cap (N(j) \cup N(i))$ that is not included in $R(i,j)$.
For simplicity of notation, let $r_j = r(U_j)$. Recalling the definition of $I_{true}$ in \eqref{Itrue} , $U_k \in I_{true}$, so
$U_k \geq r_j - d_1$. Each $l\in R(i, j)$ is a neighbor of $i$ or $j$, so $U_l < \max\{r_j,r_i\}=r_j$.
Moreover, since $l \in R(i, j)$ and $k \notin R(i, j)$, it follows from the definition of
$R(i, j)$ that $\sigma_1(l) > \sigma_1(k)$. We assumed that $\sigma_1$ agrees with $\sigma_{true}$ at precision
level $d_1$. So then either $l$ and $k$ are correctly ordered by $\sigma_1$ (in which case
$U_l > U_k \geq r_j - d_1$), or $|U_l - U_k| < d_1$ (in which case $U_l > U_k - d_1 \geq r_j - 2d_1$). It follows that for each $l \in R(i, j)$,
$r_j - 2d_1 < U_l < r_j$. Since we have that $R(i, j) \subset V_1$ and $C_1 = \left\lceil np_1 d_1\log(n) \right\rceil$,
\begin{align*}
\{ Dist(i,j)\cap (N(j) & \cup N(i)) \not \subset R(i,j)\} \\
&\subset  \{|\{s\in V_1: r_j-2d_1 <U_s<r_j\}|\geq C_1\}\subset  \mathcal{A}_1(p_1).
\end{align*}
Then we complete the proof by the fact that $(\mathcal{A}_1(p_1))^c$ occurs w.e.p.
\end{proof}

For a set $S\subset V$, an ordering $\sigma$ on $S$ and a pair $k,l\in S$, define the interval $I(S,\sigma,k,l)=\{s\in S: \sigma(k)<\sigma(s)<\sigma(l)\}$ and define $\mathcal{I}(S,\sigma)=\{I(S,\sigma,k,l): k,l\in S\}$ to be the collection of such intervals. For any vertex $i\in V$ and the set $I\subset V$, define the expected number of neighbors in $I$  
\be \label{EqDefWNeighbourSize}
W(i,I)=\sum_{k\in I}w(U_i,U_k)
\ee 
and the associated difference $\bar{W}(i,j,I)=W(j,I)-W(i,I)$. Thus $\bar{W}$ is the conditional mean of the corresponding edge-count difference after the latent positions have been fixed. Given a set $S \subseteq V$ and ordering $\sigma$ of $S$, we define for $I\in \mathcal{I}(S,\sigma)$
\begin{align*}
\mathcal{A}_2(i,j,I):=\Big\{ \Big|  |I \cap N(j)|-|I \cap N(i)| & - \bar{W}(i,j,I)  \Big|  \\
& >\sqrt{\bar{W}(i,j,I)}\log n + \log(n)^2 \Big\}.
\end{align*}
For $i,j\in V_2\setminus V_1$ and $I_U=[x,y]\subset [0,1]$, define the associated population quantity
\[
\mathcal{W}(i,I_U)=np_1\int_{x}^{y}w(U_i,z)dz
\]
and
\[
\mathcal{W}(i,j,I_U)=\mathcal{W}(j,I_U)-\mathcal{W}(i,I_U).
\]
Let $\mathcal{H}_1$ be the $\sigma$-algebra generated by all latent positions, all variables $B_k$, and the graph induced by $V_1$.

Recall the definition of ``signal" in Equation \eqref{EqDefSignal}. We have:

\begin{lemma}\label{LowerBoundSignal}
 W.e.p. for all $i,j\in V_2\setminus V_1$ satisfying $U_j-U_i>d_2$, 
\[
Signal(i,j) \geq 2C_2.
\]
\end{lemma}

\begin{proof}
From Lemma \ref{SignalBelongstoR} above, we know that w.e.p.
\be
\nonumber Signal(i,j) &= |R(i,j)\cap Dist(i,j)\cap N(j)|- |R(i,j)\cap Dist(i,j)\cap N(i) |\\
\label{SignalLowerBound} &= |Dist(i,j)\cap N(j)|- |Dist(i,j)\cap N(i)|.
\ee
Conditional on $\mathcal{H}_1$, the set $Dist(i,j)$ is fixed and the right hand side of \eqref{SignalLowerBound} has conditional mean $\bar{W}(i,j,Dist(i,j))$.

We first compare this conditional mean to the population quantity $\mathcal{W}(i,j,I_{true})$. Set
\[
\mathcal{A}_{3}(i,j)=\left\{
\left|\bar{W}(i,j,Dist(i,j))-\mathcal{W}(i,j,I_{true})\right|>C_2/2
\right\}.
\]
Conditionally on $U_i,U_j$, the random variable $\bar{W}(i,j,Dist(i,j))$ is a sum of independent bounded random variables
\[
Y_k=\mathbf{1}_{\{B_k\leq p_1\}}\mathbf{1}_{\{U_k\in I_{true}(i,j)\}}
\left(w(U_j,U_k)-w(U_i,U_k)\right), \qquad k\notin\{i,j\}.
\]
Moreover, $\mathbb{E}[\sum_{k\neq i,j}Y_k\mid U_i,U_j]$ is equal to $\mathcal{W}(i,j,I_{true})$ up to replacing $n$ by $n-2$, a harmless difference that is absorbed into the constants below. Thus, by Bernstein's inequality and a union bound over $i,j$, $(\mathcal{A}_{3}(i,j))^c$ holds w.e.p. for all pairs under consideration.

For sufficiently large $n$ satisfying $r(U_j)-U_j>d_1$, part (3) of Assumption \ref{DecayRateAssume} guarantees that for $U_j-U_i>d_2$ and $x\in [0,d_1]$,
\[
w(U_j, r(U_j)-x)-w(U_i,r(U_j)-x) \geq M_1^{-1}\left(x^{\alpha}-\max\{0,(x-d_2)^{\alpha}\}\right).
\]
Hence,
\be
\nonumber \mathcal{W}(i,j,I_{true}) &\geq  \frac{np_1}{M_1} \int_0^{d_1} \left(x^{\alpha}-\max\{0,(x-d_2)^{\alpha}\} \right)dx\\
\nonumber &=  \frac{np_1}{M_1}\left( \int_0^{d_2} x^{\alpha}dx+  \int_{d_2}^{d_1 }(x^{\alpha}-(x-d_2)^{\alpha})dx\right)\\
\nonumber &= \frac{np_1(d_1 )^{1+\alpha}}{M_1(1+\alpha)} \left[1-\left(1-\frac{d_2}{d_1}\right)^{1+\alpha}\right]\\
\nonumber &= \frac{np_1(d_1 )^{1+\alpha}}{M_1(1+\alpha)} \left[(1+\alpha)\frac{d_2}{d_1}-\frac{(1+\alpha)\alpha}{2}\left(\frac{d_2}{d_1}\right)^2+O\left(\left(\frac{d_2}{d_1}\right)^3\right)\right]\\
\nonumber &\geq \frac{1}{M_1} n p_1 d_2 (d_1)^{\alpha}\left(1-\frac{\alpha d_2}{2d_1}\right) \\
\label{SignalExpectation} &\geq \frac{1}{2M_1}np_1 d_2 (d_1)^{\alpha}=3C_2,
\ee
where the last line follows from the fact that $d_2/d_1 < 1$ and $\alpha<1$.
On $(\mathcal{A}_{3}(i,j))^c$, this implies $\bar{W}(i,j,Dist(i,j))\geq 5C_2/2$. Therefore, for $n$ sufficiently large,
\[
\sqrt{\bar{W}(i,j,Dist(i,j))}\log n+\log(n)^2 \leq \bar{W}(i,j,Dist(i,j))-2C_2.
\]
Recalling the definition of ``signal" in Equation \eqref{EqDefSignal}, we conclude that
\[
\{Signal(i,j)<2C_2\} \subset \mathcal{A}_{3}(i,j)\cup \mathcal{A}_2(i,j,Dist(i,j)).
\]
Finally, conditional on $\mathcal{H}_1$, the difference $|Dist(i,j) \cap N(j)|-|Dist(i,j) \cap N(i)|$ is a sum of independent bounded random variables with conditional mean $\bar{W}(i,j,Dist(i,j))$. Applying Hoeffding's inequality and a union bound over $i,j$, we have that
\[
\left(\mathcal{A}_2(i,j,Dist(i,j))\right)^c \mbox{ holds w.e.p. for all pairs under consideration.}
\]
Combining this estimate with $(\mathcal{A}_{3}(i,j))^c$, the result follows.
\end{proof}

For sets $S,T \subseteq V$ and ordering $\sigma$ of $S$, we define
\begin{align*}
\mathcal{A}_4(S,T,\sigma):=\bigcup_{I\in \mathcal{I}(S,\sigma),i\in T} \Big\{\Big|  |I \cap N(i)|& -W(i,I)  \Big|\\
&>\sqrt{W(i,I)}\log(n)+\log (n)^2 \Big\}.
\end{align*}

Recall the definition of ``noise" in Equation \eqref{EqDefNoise}. We have:

\begin{lemma}\label{UperBoundNoise}
W.e.p. for all $i,j\in V_2\setminus V_1$ satisfying $U_j-U_i>d_2$,
\[
Noise(i,j)\leq C_2.
\]
\end{lemma}

\begin{proof}
We first show that w.e.p. $U_j<U_k$ for all $k\in R(i,j)$. Consider the interval $J=\{s\in V_1: U_j<U_s<r(U_j)\}$ and define its latent counterpart $J_U=[U_j,r(U_j)]$. 
In view of Assumption \ref{DecayRateAssume}(1), there exists a constant $K\in (0,\infty)$ such that
\begin{align*}
\mathcal{W}(j,J_U) &= n p_1 \int_{U_j}^{r(U_j)} w(U_j,x)dx \\
&= n p_1 \int_{0}^{r(U_j)-U_j} w(U_j,r(U_j)-x)dx\\
&\geq K np_1 \int_0^{r(U_j)-U_j} x^{\alpha} dx \\
&= K n p_1(r(U_j)-U_j)^{1+\alpha} \geq K \rho^{1+\alpha} n p_1.
\end{align*}
Define the event
\[
\mathcal{A}_5(j)= \left\{
\left|W(j,J)-\mathcal{W}(j,J_U)\right|>\frac{K}{2}\rho^{1+\alpha} n p_1 \right\}.
\]
On $(\mathcal{A}_5(j))^c$, this implies $W(j,J)\geq \frac{K}{2}\rho^{1+\alpha}n p_1$. Therefore, for $n$ sufficiently large,
\[
\sqrt{W(j,J)}\log n+\log(n)^2 \leq W(j,J)-C_1,
\]
where we recall that $C_1= \left \lceil np_1 d_1 \log(n) \right \rceil \leq \left \lceil n p_1 \log(n)^{-1} \right \rceil$. Consequently,
\[
\{|N(j)\cap J|\leq C_1\} \subset \mathcal{A}_4(V_1,V_2\setminus V_1, \sigma_{true}) \cup \mathcal{A}_5(j).
\]
Using arguments similar to those used to establish that $(\mathcal{A}_2)^c$ and $(\mathcal{A}_3)^c$ hold w.e.p.,
it follows from Bernstein's inequality and Hoeffding's inequality that 
\[
\mbox{ $\left(\mathcal{A}_5(j)\right)^c$ and $\left(\mathcal{A}_4(V_1,V_2\setminus V_1, \sigma_{true})\right)^c$ }
\]
also hold w.e.p.
Hence, $\{|N(j)\cap J|> C_1\}$ occurs w.e.p.,
which implies that w.e.p. $j$ has at least $C_1$ neighbors with $U$-values greater than $U_j$. It follows that $U_j<U_k$ for all $k\in R(i,j)$.

Since $w$ is Robinson, it follows that
$W(i, R(i, j)) \leq W(j, R(i, j)) \leq |R(i, j)| = C_1$. Moreover, by definition of $R(i, j)$,
there must be an interval $I \in \mathcal{I}(V_1, \sigma_1)$ so that $R(i, j) \setminus Dist(i,j) = (I \cap N(i))\cup (I \cap N(j))$. Let $D_i=|N(i)\cap I|-W(i,I)$ and $D_j=W(j,I)-|N(j)\cap I|$. Then
\begin{align*}
Noise(i,j)&= |\left(R(i,j)\setminus Dist(i,j) \right)\cap N(i)|-|\left(R(i,j)\setminus Dist(i,j)\right) \cap N(j)|\\
&= |N(i)\cap I|-|N(j)\cap I|\leq D_i+D_j.
\end{align*}
It follows that
\begin{align*}
\{Noise(i,j)&>C_2\}\\
&\subset \{D_i +D_j \geq C_2\}\\
&\subset \{D_i\geq C_2/2\} \cup \{D_j \geq C_2/2\}\\
&\subset \{D_i\geq \sqrt{C_1}\log(n)+\log(n)^2\} \cup \{D_j \geq \sqrt{C_1}\log(n)+\log(n)^2\}\\
&\subset \mathcal{A}_4(V_1,V_2\setminus V_1, \sigma_1),
\end{align*}
where the third line comes from the fact that $C_2/2 \geq \sqrt{C_1}\log(n)+\log(n)^2$ for sufficiently large $n$. Since w.e.p. $\left(\mathcal{A}_4(V_1,V_2\setminus V_1, \sigma_1)\right)^c$ occurs, the result follows immediately.
\end{proof}

For sets $S\subseteq V$, parameter $d\in (0,1)$ and ordering $\sigma$ of $S$, we define
\[
\mathcal{A}_6(S,\sigma,d)=\bigcup_{i,j\in S,|U_j-U_i|>d}\{\mathbf{1}_{\sigma(i)<\sigma(j)}=\mathbf{1}_{\sigma_{true}(i)>\sigma_{true}(j)}\}.
\]
That is, $\left(\mathcal{A}_6(S,\sigma,d)\right)^c$ holds if $\sigma$ agrees with $\sigma_{true}$ at precision level $d$.

\begin{lemma} \label{LemmaStepTenCorrect}
 Assume that $\sigma_1$ is an ordering of $V_1$, derived only from $G_1=G(V_1)$, which agrees with $\sigma_{true}$ at precision level $d_1$. Then, for the ordering $\sigma_2'$ obtained in Step 10 of Algorithm \ref{alg:erss}, 
\[
\mathbb{P}[\mathcal{A}_6(V_2\setminus V_1,\sigma_2',d_2)| \mathcal{H}_1]=n^{-\Omega(\log(n))}.
\]
That is, w.e.p. $\sigma_2'$ is an ordering of $V_2\setminus V_1$ that agrees with $\sigma_{true}$ at precision level $d_2$.
\end{lemma} 

\begin{proof}
Assume that $i,j\in V_2\setminus V_1$ and without loss of generality $U_j-U_i>d_2$. We conclude from \eqref{Signal-Noise}, Lemmas \ref{LowerBoundSignal} and \ref{UperBoundNoise} that w.e.p.
\[
|N(j)\cap R(i,j)|-|N(i)\cap R(i,j)|= Signal(i,j)-Noise(i,j)\geq C_2,
\]
and thus $F(i,j)$ is set to 1 in Algorithm \ref{alg:erss}. By applying similar arguments to the left region, we see then that w.e.p. 
\[
|N(i)\cap L(i,j)|-|N(j)\cap L(i,j)|\geq C_2,
\]
and thus $F(i,j)$ is set to 1 in Algorithm \ref{alg:erss}. Therefore, in both cases w.e.p., $i$ and $j$ will be ordered in $\sigma_2'$ to align with $\sigma_{true}$. The result immediately follows by a union bound.
\end{proof}

Finally, we show that the extension of $\sigma_{2}'$ to all of $V_{2}$, denoted by $\sigma_{2}$, has the desired precision level.

\begin{lemma} \label{LemmaExtCorrect}
Suppose $\sigma_{2}^{\prime}$ is an ordering of $V_{2}\setminus V_{1}$ that agrees with $\sigma_{\text{true}}$ at precision level $d_{2}$. Then, for the output $\sigma_2$ of Algorithm \ref{alg:erss},
\[
\mathbb{P}\left[\bigcup_{i,j\in V_{2},\,U_{j}>U_{i}+d_{2}\log(n)^{2}}\{t(j)-t(i)<C_{3}\}\cap\{b(j)-b(i)<C_{3}\}\right]=n^{-\Omega(\log(n))}.
\]
That is, w.e.p. any ordering of $V_{2}$ extending $\sigma_{2}^{\prime}$ and based on the function $F^{(2)}$ as computed in Algorithm \ref{alg:erss} agrees with $\sigma_{\text{true}}$ at precision level $d_{2}\log(n)^{2}$.
\end{lemma}

\begin{proof}
Fix $i,j \in V_2$ satisfying $U_j-U_i>d_2 \log(n)^2$. Then we have that $\sigma_{true}(i)<\sigma_{true}(j)$. Since $w$ is Robinson,
\begin{align*}
\{z:w(U_{i},z)=0\neq w(U_{j},z)\text{ or } & w(U_{j},z)=0\neq w(U_{i},z)\}\\
&=[r(U_{i}),r(U_{j})]\cup[\ell(U_{i}),\ell(U_{j})].
\end{align*}
From Assumption \ref{U-r-relationship} we have that
$$
\left(r(U_{j})-r(U_{i})\right)+\left(\ell(U_{j})-\ell(U_{i})\right)\geq B|U_{j}-U_{i}|.
$$
If
$$
r(U_{j})-r(U_{i})\geq\ell(U_{j})-\ell(U_{i}),
$$
then
\be\label{rDifference}
r(U_{j})-r(U_{i})\geq \frac{B}{2}|U_{i}-U_{j}| > 2 d_{2}\log(n),
\ee
where the second inequality holds for all $n$ large enough that $\log(n)>\frac{4}{B}$. Define
\[
I_2(j)=\{s\in V_2 \setminus V_1: U_s>r(U_j)-d_2 \log (n)\}.
\]
We will first show that, for all $s \in I_{2}(j)$, $\sigma_{2}^{\prime}(s) > t(i)$. Let $k \in N(i)$ be the vertex at the top of the range of $N(i)$ according to $\sigma_{2}^{\prime}$, that is, $\sigma_{2}^{\prime}(k) = t(i)$. Since $k \in N(i)$, we have that $U_{k} \leq r(U_{i})$. In view of \eqref{rDifference}, we have for $s \in I_{2}(j)$,
$$
U_{s} - U_{k} \geq (r(U_{j}) - d_{2}\log(n)) - r(U_{i}) > d_{2}.
$$
This implies that $s$ and $k$ are correctly ordered by $\sigma_{2}^{\prime}$ and thus $t(i)=\sigma_{2}^{\prime}(k)<\sigma_{2}^{\prime}(s)$ as desired.

This correct ordering implies that $k\not\in I_{2}(j)$, and so we see that $t(j)-t(i)\geq|I_{2}(j)\cap N(j)|$, so we have the containment
\be\label{C3relation}
\{t(j)-t(i)<C_{3}\}\subset\{|I_{2}(j)\cap N(j)|<C_{3}\}.
\ee
Recall the definition of $W$ from Equation \eqref{EqDefWNeighbourSize}. Since we have already shown that $\mathbb{P}[\mathcal{A}_{4}(V_{2},V_{2}\setminus V_{1},\sigma_{\text{true}})]=n^{-\Omega(\log(n))}$, then it holds w.e.p. that
$$
|I_{2}\cap N(j)|\geq\frac{1}{2} W(j,I_{2}(j))\geq\frac{n(p_2-p_1)}{2(1+\alpha)} (d_2 \log(n))^{1+\alpha}>C_3,
$$
where the last inequality holds for all $n$ large enough that $p_2>2p_1$.
In view of \eqref{C3relation}, we have $t(j)-t(i)\geq C_{3}$ holds w.e.p., and thus $F^{(2)}(i,j)$ is set to 1 in Algorithm \ref{alg:erss}.

Similarly, if $U_{i}<U_{j}$ and
\[
r(U_{j})-r(U_{i})\leq \ell(U_{j})-\ell(U_{i}),
\]
then $\mathbb{P}[\{b(j)-b(i)<C_{3}\}]= n^{-\Omega(\log(n))}$. Therefore, w.e.p. $F^{(2)}(i,j)$ is set to 1 in Algorithm \ref{alg:erss}. The result follows by a union bound.
\end{proof}

All the results in this subsection lead to the following theorem: 

\begin{theorem}\label{theorem-erss}

Assume that $\sigma_1$ is an ordering of $V_1$, which agrees with $\sigma_{true}$ at precision level $d_1$. Then, for the output $\sigma_2$ of Algorithm \ref{alg:erss},
\[
\mathbb{P}[\mathcal{A}_6(V_2,\sigma_2,(n p_1)^{-1/2}(d_1)^{(1-\alpha)/2}\log(n)^2)]=n^{-\Omega(\log(n))}.
\]
That is, w.e.p. $\sigma_2$ is a total order of $V_2$ that agrees with $\sigma_{true}$ at precision level $(n p_1)^{-1/2}(d_1)^{(1-\alpha)/2}\log(n)^2$.
\end{theorem}

\subsection{Multistep Analysis} \label{SubsecOrdMany}
Fix $0 < \epsilon < 0.5$ and $0< \gamma\leq 1$, and define the following functions of the graph size $n$ for the remainder of this section:
\be
\label{parameter2} k &= \lfloor -\log_{2}(\epsilon) \rfloor + 1, \\
\nonumber \beta &= \frac{\epsilon - 2^{-k}}{k}, \\
\nonumber  p_i &= n^{-(k-i)\beta}, \quad \text{for } i = 1,\ldots,k, \\
\nonumber  d_1 &= n^{-\gamma},\\
\nonumber  d_{i+1} &= (n p_i)^{-\frac{1}{2}} (d_i)^{\frac{1-\alpha}{2}}\log(n)^2, \quad \text{for } i = 1,\ldots,k-1.
\ee
These parameters control the sequence of progressively refined estimates in our algorithm. The parameter $\epsilon > 0$ is user-specified, and our theory works with any choice of $\epsilon$ that is small enough. The parameter $\gamma$ is also user-specified, but our theory requires that the initial ordering given by the coarse algorithm has precision level $n^{-\gamma}$. The integer $k$ determines the number of iterations, which will increase as the desired error $\epsilon$ decreases, while $\beta$ ensures proper scaling between iterations. The parameters $\{p_i\}$ represent sampling probabilities, and $\{d_i\}$ control the precision thresholds at each stage. Specifically, $d_1$ corresponds to the expected precision level of the coarse algorithm $P$, and for $i\geq 2$, $d_{i+1}$ denotes the precision level to be achieved after the $i$-th refinement round.

We now show that, if Algorithm \ref{alg:eras} is executed with parameters as defined in Equation \eqref{parameter2},
and given as input a large enough graph $G$, then w.e.p. the algorithm will return a total order of $V$ with error at most $n^{\frac{\alpha}{1+\alpha}+\delta}$, where any $\delta >0$ can be chosen by the user by adjusting the value of $\epsilon  = \epsilon(\delta)$ (see Equation \eqref{epsilon_0}). 

We begin by checking that the parameters in Equation \eqref{parameter2} satisfy the conditions for Theorem \ref{theorem-erss} stated in Equation \eqref{parameter}. Recall that $\alpha$ is the decay rate of the graphon of interest, per Assumption \ref{DecayRateAssume}:

\begin{lemma}\label{UpperboundOnd_k}
For parameters as defined in \eqref{parameter2} and sufficiently large $n$, we have for $1\leq i \leq k$,
\[
d_{i}=\left(n^{-\frac{1}{2}\left[ \left(2\gamma-2\left(\frac{1-k\beta}{1+\alpha}\right)+\frac{4\beta}{(1+\alpha)^2}\right)\left(\frac{1-\alpha}{2}\right)^{i-1}+\left(\frac{2\beta}{1+\alpha}\right)i+ 2\left(\frac{1-k\beta}{1+\alpha}\right)-\frac{4\beta}{(1+\alpha)^2}\right]} \right) \log(n)^{\frac{4}{1+\alpha}\left(1-(\frac{1-\alpha}{2})^{i-1}\right)}.
\]
In particular, for any $\delta>0$, let $\epsilon_{0} = \epsilon_{0}(\delta)$ be the solution in $\epsilon$ to:

\be\label{epsilon_0}
\frac{1}{1+\alpha}\left(\frac{1-\alpha}{2}\right)^{k-1}+\frac{2^{-k}}{1+\alpha}\left(1+\frac{2}{k(1+\alpha)}\right)=\frac{\delta}{2},
\ee
where we recall that $k = k(\epsilon)$ depends on $\epsilon$. Then for any $\epsilon \leq \epsilon_0(\delta)$, we have $d_k \leq n^{-\frac{1}{1+\alpha}+\delta}$.
\end{lemma}

\begin{proof}
For $i\geq 1$, set $d_i=n^{-\frac12 g_{i}} \log (n)^{h_i}$. The following recursive definitions of the sequences $\{g_i\}$ and $\{ h_i\}$ follow directly from  definitions of $d_i$ and $p_i$ as given in \eqref{parameter2}:
\begin{align*}
g_{i+1} &= 1-(k-i)\beta+g_i\left(\frac{1-\alpha}{2}\right),\quad g_1=2\gamma,\\
h_{i+1} &= h_i\left(\frac{1-\alpha}{2}\right) + 2,\quad  h_1=0.
\end{align*}
Solving these linear recurrences via standard methods we obtain:
\begin{align*}
g_{i} &= \left(2\gamma-2\left(\frac{1-k\beta}{1+\alpha}\right)+\frac{4\beta}{(1+\alpha)^2}\right)\left(\frac{1-\alpha}{2}\right)^{i-1}\\
&\qquad \qquad \qquad +\left(\frac{2\beta}{1+\alpha}\right)i+ 2\left(\frac{1-k\beta}{1+\alpha}\right)-\frac{4\beta}{(1+\alpha)^2},\\
h_{i} &= \frac{4}{1+\alpha}\left(1-\left(\frac{1-\alpha}{2}\right)^{i-1}\right)\leq \frac{4}{1+\alpha}.\\
\end{align*}

In view of the definitions of $k,\beta$ in \eqref{parameter2},
\[
k\beta=\epsilon-2^{-k} \in (0,2^{-k}],
\quad
\beta\leq \frac{2^{-k}}{k}.
\]
It then follows that
\begin{align*}
g_k &>  \left( 2\left(\frac{1-k\beta}{1+\alpha}\right)-\frac{4\beta}{(1+\alpha)^2}\right)\left(1-\left(\frac{1-\alpha}{2}\right)^{k-1}\right) \\
& \geq  \frac{2}{1+\alpha}\left( (1-2^{-k})\left(1-\left(\frac{1-\alpha}{2}\right)^{k-1}\right)-\frac{2\beta}{1+\alpha}\right)\\
&\geq   \frac{2}{1+\alpha}\left( \left(1-\left(\frac{1-\alpha}{2}\right)^{k-1}\right)-2^{-k}-\frac{2^{1-k}}{k(1+\alpha)}\right)\\
&= \frac{2}{1+\alpha} -2\left( \frac{1}{1+\alpha}\left(\frac{1-\alpha}{2}\right)^{k-1}+\frac{2^{-k}}{1+\alpha}\left(1+\frac{2}{k(1+\alpha)}\right)\right)
\end{align*}

Since $k = k(\epsilon)$ increases as $\epsilon$ decreases, the quantity
$$
\frac{1}{1+\alpha}\left(\frac{1-\alpha}{2}\right)^{k-1}+\frac{2^{-k}}{1+\alpha}\left(1+\frac{2}{k(1+\alpha)}\right)
$$  
decreases as $\epsilon$ decreases.  Thus, it is at most $\frac{\delta}{2}$ for $\epsilon \leq \epsilon_0(\delta)$. Hence, for any $\epsilon \leq \epsilon_0(\delta)$,
\[
d_k \leq n^{-\frac{1}{1+\alpha}+\delta}(n^{-\frac{\delta}{2}}\log(n)^{\frac{4}{1+\alpha}}).
\]
The result then follows for $n$ large enough so that $n^{-\frac{\delta}{2}}\log(n)^{\frac{4}{1+\alpha}}\leq 1$.
\end{proof}

In particular, substituting the parameters to obtain the value of $d_2$, we see that after only one iteration, the ordering has reached precision close to $n^{-1/2}$, even if the initial ordering has very low precision ($\gamma$ close to zero). Therefore, the error in the initial ordering does not feature in our choice of parameters as given in \eqref{parameter2}. 

\begin{corollary}
  Assume $\sigma_1$ is an ordering of $V_1$ which agrees with $\sigma_{true}$ at precision level $d_1=n^{-\gamma}$ ($0<\gamma\leq 1$), and assume that $n$ is large enough. Then the ordering obtained after one iteration of Algorithm \ref{alg:erss} agrees with $\sigma_{true}$ at precision level $d_2 \leq n^{-1/2(1-\epsilon/2)} $.
\end{corollary}

The next lemma shows that Algorithm \ref{alg:erss}, when called by Algorithm \ref{alg:eras} to extend the ordering from $V_{i}$ to $V_{i+1}$, improves the precision from $d_{i}$ to $d_{i+1}$.

\begin{lemma}
Fix $\epsilon>0$ and let $1 \leq i < k$. Suppose $\sigma_{i}$ is an ordering of $V_{i}$ which agrees with $\sigma_{\text{true}}$ at precision level $d_{i}$, and which depends only on $G_i$. Let $\sigma_{i+1}$ be the ordering returned by Algorithm \ref{alg:erss} as it is called in Algorithm \ref{alg:eras}. Then w.e.p. $\sigma_{i+1}$ agrees with $\sigma_{\text{true}}$ at precision level $d_{i+1}$.
\end{lemma}

\begin{proof}
Fix $1 \leq i < k$ and assume $\sigma_{i}$ agrees with $\sigma_{\text{true}}$ at precision level $d_{i}$. Note that the parameters $C_{1}, C_{2}, C_{3}$ as set in Algorithm \ref{alg:eras} and used in the call to Algorithm \ref{alg:erss}, satisfy the conditions \eqref{parameter}, with $p_{i}$ and $d_{i}$ taking the role of $p_{1}$ and $d_{1}$, respectively. Theorem \ref{theorem-erss} then shows that w.e.p. $\sigma_{i+1} = SingleStage(G_{i+1}, V_{i}, \sigma_{i})$ agrees with $\sigma_{\text{true}}$ at precision level 
$$(n p_i)^{-1/2}(d_i)^{(1-\alpha)/2}\log(n)^2 = d_{i+1}.$$
\end{proof}

The final ingredient we need is a bound on the ordering error implied by the precision level.

\begin{lemma}\label{PreciseLevel-Error}
If an ordering $\sigma$ agrees with $\sigma_{\text{true}}$ at precision level $d$, then w.e.p. it has error $\mathfrak{D}$ less than $nd\log(n)$.
\end{lemma}

\begin{proof}
Fix $i,j\in V$ such that
\[
\sigma_{\text{true}}(i) > \sigma_{\text{true}}(j) + nd\log(n).
\]
Then there are at least $nd\log(n)$ indices $k\in V$ satisfying $U_j < U_k < U_i$, i.e.,
\[
|\{k\in V: U_j < U_k < U_i\}| \ge nd\log (n).
\]
Suppose $U_i - U_j < d$. Then, for all sufficiently large $n$,
\[
\{|\{k\in V: U_j < U_k < U_i\}| > nd\log(n)\} \subset \mathcal{A}_1(1).
\]
Since w.e.p.\ $(\mathcal{A}_1(1))^c$ holds, it follows that w.e.p.
\[
|\{k\in V: U_j < U_k < U_i\}| \le nd\log(n),
\]
which is a contradiction. Hence $U_i - U_j \ge d$.
By the precision assumption, we conclude that
\[
\sigma(i) > \sigma(j).
\]
Therefore, $\sigma$ has error less than $nd\log(n)$.
\end{proof}

Finally, we establish the main theorem for Algorithm \ref{alg:eras}.

\begin{theorem}\label{theorem-eras}
Suppose that the output of the coarse algorithm $P$ agrees with $\sigma_{true}$ at precision level $n^{-\gamma}$, where $0< \gamma \leq 1$.
For any $\delta>0$ and sufficiently large $n$, there exists $\epsilon_0(\delta)$ as defined in \eqref{epsilon_0} such that, when Algorithm~\ref{alg:eras} is executed with any $\epsilon \leq \epsilon_0(\delta)$ and parameters given by \eqref{parameter2} on input $G$, then the output is a permutation $\sigma$ on $\{1,2,\ldots,n\}$ with error
\[
\mathfrak{D} = O\left(n^{\frac{\alpha}{1+\alpha}+2\delta}\right)
\]
w.e.p.
\end{theorem}

\begin{proof}
By definition, $V_{k}=V$, and Algorithm \ref{alg:eras} returns the ordering $\sigma=\sigma_{k}$ of $V$. By the previous lemmas and a union bound, w.e.p. $\sigma$ will agree with $\sigma_{\text{true}}$ (or its reverse) at precision level $d_{k}$. By Lemma \ref{UpperboundOnd_k}, for any $\delta>0$ one can choose $\epsilon\leq \epsilon_0(\delta)$ small enough so that $d_k \leq n^{-\frac{1}{1+\alpha}+\delta}$.

It is shown in Lemma \ref{PreciseLevel-Error} that, w.e.p., if an ordering agrees with $\sigma_{\text{true}}$ at precision level $d$, then it has error less than $nd\log(n)$. Since $n^{-\delta}\log(n)<1$ for large enough $n$, it follows that w.e.p. $\sigma$ has error at most $n^{\frac{\alpha}{1+\alpha}+2\delta}$.
\end{proof}

\section{Applications} \label{SecApplications}

We show that a good algorithm for estimating the ordering can be used to provide statistically- and computationally-efficient algorithms for two downstream tasks: graphon estimation (Section \ref{SecGraphEst}) and property testing (Section \ref{SecPT}).

\subsection{Application to Graphon Estimation} \label{SecGraphEst}

Informally, the \textit{graphon estimation problem} is: given an observed random graph $G$, estimate the graphon $w$ that $G$ was sampled from. Formalizing this problem is not trivial, as there are many symmetries that make the graphon itself non-identifiable given a sampled graph. In this paper, we follow \cite{gao_rog}: rather than estimating the unidentifiable graphon $w$ itself, we construct an estimate $\tilde{\theta}_{i,j}$ of the identifiable values $w(U_{i},U_{j})$ at the unobserved latent vertex positions. We measure error with the usual $L^{2}$ loss:

\be \label{EqGraphonEstProblem}
\frac{1}{n^{2}} \sum_{i,j = 1}^{n}(\tilde{\theta}_{i,j} - w(U_{i},U_{j}))^{2}.
\ee

In \cite{gao_rog}, the authors find the optimal rate of convergence for the estimation problem in Equation \eqref{EqGraphonEstProblem}. As with the main result of this paper, they find that the optimal rate of convergence depends strongly on a parameter $\alpha$ that describes how quickly derivatives of $w$ can blow up. Although \cite{gao_rog} shows that a certain estimator converges at the optimal rate, it isn't clear how to actually compute the estimator. In this section, we show how to use our estimate for the seriation problem to construct an estimate for the graphon-estimation problem, and we show that this estimate converges at essentially the same rate as the optimal estimate in \cite{gao_rog} (though we do pay a small price for using our much more computationally-tractable estimate).

\subsubsection{Overview and Main Construction}

Our graphon estimator uses the principle of ``local averaging", which is essentially the same as the strategy used in \cite{gao_rog} and described in Equation \eqref{EqWHatNaive}. Namely, for any pair of indices $(i,j)$, guess which pairs $(U_{k},U_{\ell})$ are close to $(U_{i},U_{j})$, then estimate $w(U_{i},U_{j})$ as a weighted average of the number of edges between such pairs. The main difficulty in carrying out this strategy is accurately identifying the pairs of indices $(k,\ell)$ for which $(U_{k},U_{\ell})$ is close to $(U_{i},U_{j})$. 
From the analysis in Section \ref{section-error-rooting}, we might hope that Algorithm \ref{alg:eras} largely solves this problem by giving a very accurate estimate of this set of pairs of indices. 

The main steps in our estimation method are as follows. 
\begin{enumerate}
\item Partition the vertex set $V$ into three parts $V^{(r)}$, $r\in [3]$ with \emph{complement} $V_c^{(r)}$. For each $r\in [3]$ perform the Steps (2)--(5) below. 
\item Use Algorithm \ref{alg:eras} to obtain an ordering $\sigma_{train}$ on the subgraph induced by $V^{(r)}$.
\item Use Algorithm \ref{alg:erss} to extend $\sigma_{train}$ to an ordering $\sigma_{extend}$ to all of $V$. 
\item Partition $V_c^{(r)}$ into $m$ parts of approximately equal size so that each part is an interval according to $\sigma_{extend}$. The partition is defined by a function $\tilde{z}$ which assigns vertices to parts. 
\item For any $a,b\in [m]$,  compute the density of edges between the parts  $\tilde{z}^{-1}(a)$ and $\tilde{z}^{-1}(b)$, and use this to construct an estimator $\tilde{\theta}^{(r)}$.
\item Combine the estimators $\tilde{\theta}^{(r)}$, $r=1,2,3$, into the final estimator  $\tilde{\theta} $ by averaging.
\end{enumerate}

It is worth mentioning that Algorithm \ref{alg:graphon-estimator} is slightly more complicated than one might expect given our sketch. This extra care is required to ensure that there is sufficient independence between steps of the algorithm. For example, we partition the vertices in Step (1) to ensure that the edges that are used to obtain the ordering $\sigma_{extend}$ which informs the partition are distinct from the edges that are used to construct the estimator. 
This means that the indices of summation and the actual summands in the block estimators are independent, allowing for a simpler analysis. 

To make the above precise and to describe our algorithm in detail, we first introduce the following notation.
Fix a graphon $w$ and an integer $n$. Let $V = [n]$, and consider a random graph $G = (V,E)$ together with latent position variables $\{U_i\}_{i=1}^n$, generated according to the sampling procedure described around Equation~\eqref{eq:graphon_sampling}. We define the adjacency matrix $A$ by $A_{ij} = \mathbf{1}_{(i,j)\in E}$ for $i \neq j$, and set the target $\theta_{ij} = w(U_i, U_j)$ for $i \neq j$. We adopt the convention $A_{ii} = \theta_{ii} = 0$ for all $i \in V$.

\begin{defn}\label{def:ordered_partition}
For an integer $m$, let $Z_{V,m} = \{z : V \rightarrow [m]\}$ denote the collection of all possible mappings from $V$ to $[m]$; we will refer to such mappings as \emph{partition functions}. For any $z \in Z_{V,m}$, the sets $\{z^{-1}(a) : a \in [m]\}$ form a partition of $V$. Given $m$ and a  
function $\sigma \, : \, V \mapsto \mathbb{N}$
we define the partition function $\tilde{z} = \tilde{z}[V,\sigma,m] \in Z_{V,m}$ as follows. Let $q = q(V,m) \equiv \lfloor |V|/m \rfloor$, and set
\be\label{tilde-z}
\tilde{z}(i) = \min\left\{m, \left\lceil \tfrac{\sigma(i)}{q} \right\rceil\right\}, 
    \quad i \in V.
\ee
If $\sigma$ represents an ordering of $V$ by assigning each element of $V$ a unique rank between 1 and $|V|$, then the partition induced by $\tilde{z}$ divides $V$ into $m$ disjoint subsets, each of size roughly $q = \lfloor |V|/m \rfloor$, while the last part has size at most $2m$. We refer to $q$ as the block size. Note that each block is an interval in the ordering.
\end{defn}


Given a partition function $z \in Z_{V,m}$ and a matrix $\{A_{ij} \} \in \mathbb{R}^{|V|\times |V|}$, we
use the overline notation `` $\bar{\,}$ " to denote the block average of $A_{ij}$ on the set $z^{-1}(a) \times z^{-1}(b)$. Specifically, for any $a,b\in [m]$,
\[
\bar{A}_{ab}(z)=\frac{1}{|z^{-1}(a)| |z^{-1}(b)|}\sum_{i\in z^{-1}(a)}\sum_{j\in z^{-1}(b)} A_{ij},\quad \mbox{for } a\neq b\in[m],
\]
and when $|z^{-1}(a)|>1$,
\[
\bar{A}_{aa}(z)=\frac{1}{|z^{-1}(a)| (|z^{-1}(a)|-1)}\sum_{i\in z^{-1}(a)}\sum_{\substack{j\in z^{-1}(a),\\j\neq i}} A_{ij},\quad \mbox{for } a\in[m].
\]
We adopt the convention that if $|z^{-1}(a)|\leq 1$, then $\bar{A}_{aa}=0$. 

The algorithmic procedure is detailed in Algorithm \ref{alg:graphon-estimator}.

\begin{algorithm}[h]
\caption{Graphon Estimator}\label{alg:graphon-estimator}
\begin{algorithmic}[1]
\algrenewcommand\alglinenumber[1]{}
\State \textbf{parameters:} Partition size $m$, {decay rate $\alpha \in [0,1]$, accuracy $\delta>0$} and coarse algorithm $P$ with precision parameter $\gamma$.
\State \textbf{input:} Observed graph $G = (V,E)$ with adjacency matrix $A = \{A_{ij}: i,j \in [n]\}$.
\State \textbf{output:} Graphon estimator $\tilde{\theta}$.
\algrenewcommand\alglinenumber[1]{\scriptsize #1}
\setcounter{ALG@line}{0}
\State Select a partition of $V$ into three subsets $V^{(1)}$, $V^{(2)}$ and $V^{(3)}$ of near-equal sizes $| \, |V^{(i)}| - |V^{(j)}| \, | \leq 1$, uniformly at random. For each $r \in [3]$, let $V^{(r)}_c = V \setminus V^{(r)}$ denote the complement of $V^{(r)}$ in $V$.
\For{$r \in [3]$}
  \State Construct $G_{\mathrm{train}}$ from $\{ A_{ij} : i, j \in V^{(r)} \}$ and run Algorithm \ref{alg:eras} on $G_{train}$ with $n^{(r)}=|V^{(r)}|$ and parameters $k$, $\{p_i\}_{i=1}^k$, $\{d_i\}_{i=1}^k$, $C_1$, $C_2$ and $C_3$ specified in \eqref{parameter3} to obtain the ordering $\sigma_{train}$ on $V^{(r)}$.
  \State Run Algorithm \ref{alg:erss} on the full graph $G$ with input sets $V^{(r)} \subset V$, the ordering $\sigma_{train}$ on $V^{(r)}$ and parameters $C'_1$, $C'_2$ and $C'_3$ as specified in \eqref{parameter3'} to produce the extended ordering $\sigma_{extend}$ on $V$.
  \State Define $\sigma$ by $\sigma(i)= |\{j\in V^{(r)}_c: \sigma_{extend}(j) \leq \sigma_{extend}(i)\}|$ for $i\in V^{(r)}_c$, and compute $\tilde{z} = \tilde{z}[V^{(r)}_c,\sigma,m]$ according to Equation \eqref{tilde-z}.
  \State {Let $\widetilde{Q}_{ab}(\tilde{z})=\bar{A}_{ab}(\tilde{z})$ and define} $\tilde{\theta}^{(r)}$ by
  \be\label{theta^s}
    \tilde{\theta}^{(r)}_{ij} =
    \begin{cases}
      \widetilde{Q}_{\tilde{z}(i)\tilde{z}(j)}, & i,j \in V^{(r)}_c,  \\
      0, & \mbox{otherwise}.
    \end{cases}\nonumber
  \ee
\EndFor
\State Define $\tilde{\theta}$ by 
\be\label{tilde-theta}
\tilde{\theta}_{ij}=
\sum_{r,\ell \in [3]} \mathbf{1}_{\{i\in V^{(r)}, j\in V^{(\ell)}\}} \frac{\sum_{k\in [3]\setminus \{r,\ell\}}\tilde{\theta}^{(k)}_{ij}}{\left|[3]\setminus \{r,\ell\}\right|}.
\ee
\State \Return $\tilde{\theta}$.
\end{algorithmic}
\end{algorithm}

\subsubsection{Analysis of Algorithm \ref{alg:graphon-estimator}}\label{SecGraphEstNot}

Before proceeding with the analysis, we specify the parameters used in Algorithm \ref{alg:graphon-estimator}. 
Let $\alpha$ be the decay rate of the graphon  (see Assumption \ref{DecayRateAssume}) and let $\gamma$ ($0< \gamma\leq 1$) be the precision parameter of coarse algorithm  $P$. Fix $\delta >0$ and let $\epsilon=\epsilon_0(\delta)$ be defined as in \eqref{epsilon_0}.
For $r\in[3]$, let $n^{(r)}=|V^{(r)}|$, and define the following parameters employed in Step 3 of Algorithm \ref{alg:graphon-estimator} as functions of the graph size $n^{(r)}$:
\be
\label{parameter3}  k &= \lfloor -\log_{2}(\epsilon) \rfloor + 1, \\
\nonumber \beta &= \frac{\epsilon - 2^{-k}}{k}, \\
\nonumber  p_i &= \left(n^{(r)}\right)^{-(k-i)\beta}, \quad \text{for } i = 1,\ldots,k, \\
\nonumber  d_1 &= \left(n^{(r)}\right)^{-\gamma},\\
\nonumber  d_{i+1} &= \left(n^{(r)} p_i\right)^{-\frac{1}{2}} (d_i)^{\frac{1-\alpha}{2}}\log \left(n^{(r)}\right)^2, \quad \text{for } i = 1,\ldots,k-1\\
\nonumber  C_1 &=\left\lceil n^{(r)} p_i d_i\log\left(n^{(r)}\right) \right\rceil, \\
\nonumber  C_2 &= \left\lfloor \frac{1}{6M_1} \left(n^{(r)} p_i\right)^{\frac{1}{2}}(d_i)^{\frac{1+\alpha}{2}} \right\rfloor,\\
\nonumber  C_3 &=\left\lfloor \frac{1}{2(1+\alpha)}\left(n^{(r)} p_i\right)^{\frac{1-\alpha}{2}}  (d_i)^{\frac{1-\alpha^2}{2}} \left(\log \left(n^{(r)}\right)\right)^{1+\alpha} \right\rfloor.
\ee

Moreover, the parameters used in Step 4 of Algorithm \ref{alg:graphon-estimator} are defined as
\be
\label{parameter3'} C'_1 &= \left\lceil n^{(r)}\, n^{-\frac{1}{1+\alpha}+\delta}\log(n) \right\rceil,\\
\nonumber  C'_2 &= \left\lfloor \frac{1}{6M_1}\, \left(n^{(r)}\right)^{\frac{1}{2}} n^{-\frac{1}{2}+\frac{\delta(1+\alpha)}{2}}  \right\rfloor,\\
\nonumber  C'_3 &= \left\lfloor \frac{1}{2(1+\alpha)}\,\left(n^{{(r)}}\right)^{\frac{1-\alpha}{2}} \, n^{-\frac{1-\alpha}{2}+\frac{\delta(1-\alpha^2)}{2}}  (\log (n))^{1+\alpha}\right\rfloor.
\ee

Note that the first set of parameters corresponds to the parameters defined earlier in \eqref{parameter} and \eqref{parameter2}, where $n$ is replaced by $n^{(r)}$. The definition of the $C_i'$ is equal to the definition of the $C_i$ with the substitution $p_i = n^{-\frac{1}{1+\alpha}+\delta}$.

Given an integer $m$ and an observed adjacency matrix $\{A_{ij}\}$ on $V$, for any $Q=\{Q_{ab}\}\in \mathbb{R}^{m\times m}$ and any partition function $z\in Z_{V,m}$, we define the objective function
\[
L(Q,z)=\sum_{a,b\in[m]}\sum_{\substack{i,j\in z^{-1}(a)\times z^{-1}(b),\\ i\neq j}}(A_{ij}-Q_{ab})^2.
\]
For any optimizer of the objective function
\[
(\widehat{Q},\hat{z})\in \mbox{argmin}_{Q\in \mathbb{R}^{m\times m},z\in Z_{V,m}} L(Q,z),
\]
the {$m$-}optimal estimator of $\theta_{ij}$ is defined as
\be\label{hat-theta}
\hat{\theta}_{ij}=\widehat{Q}_{\hat{z}(i)\hat{z}(j)},\quad \mbox{for }i\neq j,
\ee
and $\hat{\theta}_{ii}=0$. 
{Rather than computing this global minimizer directly, Step 6 of Algorithm~\ref{alg:graphon-estimator} constructs a block-average estimator of the form $\widetilde{Q}_{ab}(\tilde{z}) = \bar{A}_{ab}(\tilde{z})$, which is then converted into the graphon estimator $\tilde{\theta}^{(r)}$.}

{
To analyze the performance of our procedure, we first recall a result characterizing the error rate of the $m$-optimal estimator. In particular, Theorem 2.3 of \cite{gao_rog} establishes the following error scale for $\hat{\theta}$.}

\begin{lemma}\label{Gao's} Let $m=\left\lceil n^{1/(\alpha+1)} \right\rceil$, and let $\hat{\theta}$ denote the {$m$-}optimal graphon estimator on $V=[n]$ as defined in \eqref{hat-theta}. Then with high probability,
\[
\frac{1}{n^2}\sum_{i,j\in [n]} (\hat{\theta}_{ij}-\theta_{ij})^2 = O\left(n^{-2\alpha/(\alpha+1)}+n^{-1}\log n \right)= 
O(n^{-2\alpha/(\alpha+1)}).
\]
\end{lemma}

We now show that the practical graphon estimator constructed in Algorithm \ref{alg:graphon-estimator} achieves an error rate of order $n^{-2\alpha/(1+\alpha)}$, which matches that of the {$m$-}optimal estimator in Lemma \ref{Gao's}.

For the graphon-estimation bound, assume the following H\"older condition.

\begin{assumption}[H\"older Continuity]\label{assumption:Holder}
There exists a constant $M_2>0$ so that, for all $u,u',v \in [0,1]$,
\[ 
|w(u,v)-w(u',v)| \leq M_2\, |u-u'|^\alpha .
\]
\end{assumption}

For orderings of vertex subsets we use rank functions. If $S$ has total order $\prec$, its rank function is $\sigma:S\to[|S|]$, defined by
\[
\sigma(i)=1+|\{j\in S:j\prec i\}|.
\]
When $S=[n]$, this is a permutation. Definitions~\ref{def:ordering_error} and~\ref{Def-precision-level} apply with this convention.

\begin{theorem}\label{ourTheorem}
Suppose that graphon $w$ satisfies Assumptions \ref{DecayRateAssume}, \ref{U-r-relationship} and \ref{assumption:Holder}. Fix an integer $n \in \mathbb{N}$ and let $G \sim w$ on vertex set $V=[n]$. Let $\alpha >0$ be the decay rate of $w$ and let $\gamma >0$ be the precision parameter of coarse algorithm $P$.  Fix $0 < \delta < 1/(3(1+\alpha))$ and let $\epsilon = \epsilon_0(\delta)$ as defined in \eqref{epsilon_0}. Let $\tilde{\theta}$ be the output of Algorithm~\ref{alg:graphon-estimator} executed with coarse algorithm $P$, accuracy $\delta $,  $m=\lceil n^{\frac{1}{\alpha+1}-3\delta} \rceil$, and input graph $G$. 
Then $\tilde{\theta}$ satisfies
\[
\frac{1}{n^2}\sum_{i,j\in [n]} (\tilde{\theta}_{ij}-\theta_{ij})^2 = O(n^{-\frac{2\alpha}{1+\alpha}+6\alpha \delta})
\]
w.e.p.

\end{theorem}

\begin{proof}

Fix $r \in [3]$ until the final step of the proof, where we will average over all cases. 

We begin by defining several families of objects related to partitions of our vertex set. Informally, we use superscripts to indicate what rank function is being used to define the partition, and to indicate the level of averaging. 
In this part of the proof, $\sigma_{\mathrm{true}}: V_c^{(r)} \to [|V_c^{(r)}|]$ denotes the rank function induced by the true ordering after re-ranking the vertices within $V_c^{(r)}$.
Let $z^{\star}=\tilde{z}[V^{(r)}_c, \sigma_{true},m]$ be the partition function corresponding to the true ordering with ranking function $\sigma_{true}$ as defined in \eqref{tilde-z}. Specifically, we use the superscript ``$\star$" to denote quantities block-averaged under the true ordering $z^{\star}$ and the ideal graphon $\theta$ defined by $\theta_{ij}=w(U_i,U_j)$.  From \eqref{tilde-z} it follows that,
\[
(z^{\star})^{-1}(a)=\left\{i\in V^{(r)}_c: (a-1)q < \sigma_{true}(i)\leq a q\right\},\quad \mbox{for } 1\leq a \leq m-1,
\]
and
\[
(z^{\star})^{-1}(m)=\left\{i\in V^{(r)}_c: (m-1)q< \sigma_{true}(i)\leq |V^{(r)}_c| \right\}.
\]

Define $\{Q^{\star}_{ab}\}\in \mathbb{R}^{m \times m}$ by letting $Q^{\star}_{ab}=\bar{\theta}_{ab}(z^{\star})$ for any $a,b\in [m]$. Define $\theta^{\star}_{ij}=Q^{\star}_{z^{\star}(i)z^{\star}(j)}$ for all $i\neq j$, and define the diagonal elements $\theta^{\star}_{ii} = 0$ for all $i\in V^{(r)}_c$. 

Similarly, define $\{\check{Q}_{ab}\}\in \mathbb{R}^{m \times m}$ by $\check{Q}_{ab}=\bar{A}_{ab}(z^{\star})$. Also define $\check{\theta}_{ij}=\check{Q}_{z^{\star}(i)z^{\star}(j)}$ for all $i\neq j$ and set the diagonal elements $\check{\theta}_{ii} = 0$ for all $i\in V^{(r)}_c$. That is, we use the superscript `` $\check{}$ " to denote quantities block-averaged under the true-order partition $z^{\star}$ and the observed adjacency matrix $A$. Recall that the superscript `` $\tilde{}$ " denotes quantities block-averaged under the estimated-order partition $\tilde{z} $ and the observed adjacency matrix $A$.

Fix $a,b\in [m]$. Let $S^{\star}=\{(i,j)\in V^{(r)}_c\times V^{(r)}_c:i\in (z^{\star})^{-1}(a), j\in (z^{\star})^{-1}(b)\}$ denote the set of edges between the vertices corresponding to $(z^{\star})^{-1}(a)$ and $(z^{\star})^{-1}(b)$. Similarly, define $\widetilde{S}=\{(i,j) \in V^{(r)}_c\times V^{(r)}_c:i\in \tilde{z}^{-1}(a), j\in \tilde{z}^{-1}(b)\}$. The focus of our analysis will be on sets $\widetilde{S} \setminus S^{\star}$ and $S^{\star} \setminus \widetilde{S}$, which represent the pairs of vertices that are in $\tilde{S}$ but not in $S^{\star}$, and vice versa. By symmetry, the sizes of these two sets are equal: $|\widetilde{S} \setminus S^{\star}|=|S^{\star} \setminus \widetilde{S}|$. We denote this shared size by $|\Delta S|$, where $\Delta S$ denotes the set $\widetilde{S} \setminus S^{\star}$ or $S^{\star} \setminus \widetilde{S}$, depending on the context.
We omit the dependence of  $S^{\star}$, $\widetilde{S}$ and $\Delta S$ on $a,b$ from the notation whenever it is clear from the context. 

For any $\{A_{ij}\}\in \mathbb{R}^{|V^{(r)}_c|\times |V^{(r)}_c|}$, we denote  by $\|A\|_2=\sqrt{\sum_{i,j\in V^{(r)}_c}A_{ij}^2}$ the usual $l_2$-norm. The error in our estimator can then be expressed as $\|\tilde{\theta}^{(r)}-\theta\|_2^2$. By the triangle inequality, we have
\be\label{main_inequality}
\|\tilde{\theta}^{(r)}-\theta\|_2 \leq \|\tilde{\theta}^{(r)}-\check{\theta}\|_2 + \|\check{\theta}-\theta^{\star}\|_2 + \|\theta^{\star}-\theta\|_2.
\ee

The bulk of our proof consists of bounding the first term on the RHS of \eqref{main_inequality}, which can be rewritten as
\be
\nonumber \|\tilde{\theta}^{(r)}-\check{\theta}\|^2_2 &= \sum_{i,j\in V^{(r)}_c }(\tilde{\theta}_{ij}^{(r)}-\check{\theta}_{ij})^2\\
\label{third_term_main_inequality} &=\sum_{i,j \in V^{(r)}_c}\left(\bar{A}_{\tilde{z}(i)\tilde{z}(j)}(\tilde{z})-\bar{A}_{z^{\star}(i)z^{\star}(j)}(z^{\star})\right)^2.\nonumber
\ee



We bound the sum by splitting the indices $(i,j)$ that are being summed over into three sets
\be 
L_1 & =\{i,j\in V^{(r)}_c: \tilde{z}(i)=z^{\star}(i), \tilde{z}(j) = z^{\star}(j)\}, \nonumber\\
L_{2} &= \{i,j\in V^{(r)}_c: \tilde{z}(i)=z^{\star}(i), \tilde{z}(j) \neq z^{\star}(j) \mbox{ or } \tilde{z}(i)\neq z^{\star}(i), \tilde{z}(j) = z^{\star}(j) \}, \nonumber\\
L_{3} &= \{i,j \in V^{(r)}_c: \tilde{z}(i)\neq z^{\star}(i), \tilde{z}(j)\neq z^{\star}(j)\}, \nonumber
\ee 
and bounding them individually across three cases.

The discrepancy between $\tilde{z}$ and $z^{\star}$ is driven by the difference between the underlying vertex orderings $\sigma$ and $\sigma_{\mathrm{true}}$. By Theorem \ref{theorem-eras}, the ordering $\sigma_{\mathrm{train}}$ produced by Step 3 in Algorithm~\ref{alg:graphon-estimator} achieves error $O(n^{\frac{\alpha}{1+\alpha} + 2\delta})$. By Theorem \ref{theorem-erss}, this error bound extends to $\sigma_{\mathrm{extend}}$ in Step 4, and hence to the final ordering $\sigma$. We thus define
\[
\mathfrak{D}=\max_{i \in V_c^{(r)}} |\sigma(i)-\sigma_{true}(i)| = O\left(n^{\frac{\alpha}{1+\alpha} + 2\delta}\right).
\]

Any vertex $i$ for which $\tilde{z}(i)=b $ and $z^{\star}(i)=a<b$ satisfies $\sigma_{true}(i)\leq aq \leq (b-1)q<\sigma(i)$. It follows from the definition of $\mathfrak{D}$ that
$\sigma_{true}(i)\in [aq-\mathfrak{D}, aq+\mathfrak{D}]$ for some $a\in [m]$. Thus there are at most $(m-1)(2\mathfrak{D})$ such vertices.

 \begin{enumerate} 
 \item Case 1: $\tilde{z}(i)=z^{\star}(i)=a$ and $\tilde{z}(j)=z^{\star}(j)=b$. 

Since $m=\lceil n^{\frac{1}{\alpha+1}-3\delta} \rceil$, $q=\lfloor |V_c^{(r)}|/m \rfloor=n^{\frac{\alpha}{1+\alpha}+3\delta}(1+o(1))$  and $m\mathfrak{D}=O(n^{1-\delta})$.
We note that 
\[
    (n^{(r)})^2\geq |L_1|\geq \left( n^{(r)}-(m-1)(2\mathfrak{D})\right)^2 \geq (n^{(r)})^2 - 4n^{(r)}m\mathfrak{D},
    \]
which implies $|L_1|=\Theta(n^2)$.
If $a\neq b$, then
\be
\nonumber \bar{A}_{\tilde{z}(i)\tilde{z}(j)}(\tilde{z})&-\bar{A}_{z^{\star}(i)z^{\star}(j)}(z^{\star})= \bar{A}_{ab}(\tilde{z})-\bar{A}_{ab}(z^{\star})\\
\label{case1-1} &\leq \frac{1}{q^2}\left(\sum_{i\in \tilde{z}^{-1}(a)}\sum_{j\in \tilde{z}^{-1}(b)}A_{ij}-\sum_{i\in (z^{\star})^{-1}(a)}\sum_{j\in (z^{\star})^{-1}(b)}A_{ij}\right).
\ee
If $a= b$, then
\be
\nonumber \bar{A}_{\tilde{z}(i)\tilde{z}(j)}(\tilde{z})&-\bar{A}_{z^{\star}(i)z^{\star}(j)}(z^{\star})= \bar{A}_{aa}(\tilde{z})-\bar{A}_{aa}(z^{\star})\\
\nonumber &\leq \frac{1}{q(q-1)}\left(\sum_{i\in \tilde{z}^{-1}(a)}\sum_{\substack{j\in \tilde{z}^{-1}(a),\\
j\neq i}}A_{ij}-\sum_{i\in (z^{\star})^{-1}(a)}\sum_{\substack{j\in (z^{\star})^{-1}(a),\\
j\neq i}}A_{ij}\right)\\
\label{case1-2} & = \frac{1}{q(q-1)}\left(\sum_{i,j\in \tilde{z}^{-1}(a)}A_{ij}-\sum_{i,j\in (z^{\star})^{-1}(a)}A_{ij}\right),
\ee
where the last equality comes from the convention $A_{ii}=0$ for all $i\in V^{(r)}_c$.
In view of \eqref{case1-1} and \eqref{case1-2},
\begin{align*}
\bar{A}_{\tilde{z}(i)\tilde{z}(j)}(\tilde{z})&-\bar{A}_{z^{\star}(i)z^{\star}(j)}(z^{\star})\\
&\leq \frac{1}{q^2}\left(\sum_{i\in \tilde{z}^{-1}(a)}\sum_{j\in \tilde{z}^{-1}(b)}A_{ij}-\sum_{i\in (z^{\star})^{-1}(a)}\sum_{j\in (z^{\star})^{-1}(b)}A_{ij}\right)(1+o(1)).
\end{align*}

The remainder of this case is devoted to checking that:
\be\label{MainContribution}
\left(\sum_{i\in \tilde{z}^{-1}(a)}\sum_{j\in \tilde{z}^{-1}(b)}A_{ij}-\sum_{i\in (z^{\star})^{-1}(a)}\sum_{j\in (z^{\star})^{-1}(b)}A_{ij}\right)^2 = O(n^{\frac{2\alpha}{1+\alpha}+4\alpha \delta+10\delta}).
\ee
Once this is proved, then since $|L_1|\leq n^2$ and $q=n^{\frac{\alpha}{1+\alpha}+3\delta}(1+o(1))$ we have that
\be 
\nonumber \sum_{(i,j)\in L_1} \left(\tilde{\theta}^{(r)}_{ij}-\check{\theta}_{ij}\right)^2 &=\sum_{(i,j)\in L_1} \left(\bar{A}_{\tilde{z}(i)\tilde{z}(j)}(\tilde{z})-\bar{A}_{z^{\star}(i)z^{\star}(j)}(z^{\star})\right)^2\\
\label{IneqCase1WhatWeWant} & {=\ O\left(\frac{|L_1|}{q^4}n^{\frac{2\alpha}{1+\alpha}+4\alpha \delta + 10\delta}\right)= O\left(n^{\frac{2}{1+\alpha}+4\alpha\delta-2\delta}\right).}
\ee

To quantify the contributions from the edge discrepancies involved in \eqref{MainContribution}, we define
\[
I_1=\left|\sum_{(i,j)\in S^{\star}\setminus \tilde{S}} (A_{ij}-w(U_i,U_j))\right|
\]
and
\[
I_2=\left| \sum_{(i,j)\in \tilde{S}\setminus S^{\star}}(A_{ij}-w(U_i,U_j)) \right|.
\]
Then
\be
\nonumber \Bigg|\sum_{i\in \tilde{z}^{-1}(a)} \sum_{j\in \tilde{z}^{-1}(b)} & A_{ij} -\sum_{i\in (z^{\star})^{-1}(a)}\sum_{j\in (z^{\star})^{-1}(b)}A_{ij}\Bigg| =  \left| \sum_{(i,j)\in \tilde{S} \setminus S^{\star}} A_{ij} - \sum_{(i,j)\in S^{\star} \setminus \tilde{S}} A_{ij} \right| \\
\label{sumA_ij} &\leq 
I_1+I_2+ \left|\sum_{(i,j)\in S^{\star} \setminus \tilde{S}} w(U_i,U_j) - \sum_{(i,j)\in \tilde{S} \setminus S^{\star} } w(U_i,U_j)\right|.
\ee
We begin by bounding the third term on the right-hand side. 
By Lemma~\ref{1to1-appendix} in Appendix~\ref{sec:appendix-proofs}, there exists a bijection between pairs of indices $(\tilde{i},\tilde{j})\in \tilde{S} \setminus S^{\star}$ and $(i,j)\in S^{\star} \setminus \tilde{S}$. Moreover, under this bijection, the corresponding indices satisfy
\[
|U_i - U_{\tilde{i}}| \text{ and } |U_j - U_{\tilde{j}}| = O\left(\frac{\mathfrak{D}}{n}\right) = O\bigl(n^{-\frac{1}{1+\alpha} + 2\delta}\bigr).
\]
This follows from the fact that the estimated ordering $\sigma$ agrees with $\sigma_{\text{true}}$ at precision level $n^{-\frac{1}{1+\alpha} +2\delta}$ (see Theorems \ref{theorem-eras} and \ref{theorem-erss} and their proofs). By Assumption \ref{assumption:Holder},
\begin{align*}
|w(u,v)-w(\tilde{u},\tilde{v})| & \leq |w(u,v)-w(\tilde{u},v)| +|w(\tilde{u},v)-w(\tilde{u},\tilde{v})| \\
&\leq M_2\,|u-\tilde{u}|^{\alpha}+M_2\,|v-\tilde{v}|^{\alpha}.
\end{align*}

From these observations, we can bound:
\begin{eqnarray*}
&& \left|\sum_{(i,j)\in S^{\star} \setminus \tilde{S}} w(U_i,U_j) - \sum_{(i,j)\in \tilde{S} \setminus S^{\star} } w(U_i,U_j)\right| \\
&=&  \left| \sum_{(i,j)\in S^{\star} \setminus \tilde{S}}  w(U_i,U_j) - w(U_{\tilde{i}},U_{\tilde{j}})\right| \\
&\leq&  \sum_{(i,j)\in S^{\star} \setminus \tilde{S}} \left| w(U_i,U_j) - w(U_{\tilde{i}},U_{\tilde{j}})\right| \\
& \leq&  \sum_{(i,j)\in S^{\star} \setminus \tilde{S}} \left(\left| w(U_i,U_j) - w(U_{\tilde{i}},U_j)\right|+\left| w(U_{\tilde{i}},U_j) - w(U_{\tilde{i}},U_{\tilde{j}})\right| \right)\\
&\leq&  |\Delta S| \left( M_2 |U_i-U_{\tilde{i}}|^{\alpha}+ M_2 |U_j-U_{\tilde{j}}|^{\alpha}\right)\\
&=& \ O(|\Delta S| n^{-\frac{\alpha}{1+\alpha}+2\alpha \delta}).
\end{eqnarray*}
Hence,
\be
\left(\sum_{(i,j)\in S^{\star} \setminus \tilde{S}} w(U_i,U_j) - \sum_{(i,j)\in \tilde{S} \setminus S^{\star} } w(U_i,U_j)\right)^2= O(|\Delta S|^2 \, n^{-\frac{2\alpha}{1+\alpha}+4\alpha \delta}).\nonumber
\ee
We now proceed to bound the remaining terms in Inequality \eqref{sumA_ij}. 
Given the adjacency matrix $A$ and partitions $\tilde{z}$ and $z^{\star}$, we define
\begin{align*}
\mathcal{A}_7(A,\tilde{z},z^{\star}):=\bigcup_{a,b \in [m]} \Big\{|  I_1 |>\sqrt{|\Delta S|}\log(n)+\log (n)^2 \Big\}.
\end{align*}
As discussed at the beginning of this subsection, the edges used in Step~6 of Algorithm~\ref{alg:graphon-estimator} are independent of those used to construct $\tilde{z}$. Moreover, the random variables $\{A_{ij} : i,j \in V_c^{(r)}\}$ are mutually independent. In particular, the collections $\{A_{ij} : (i,j) \in \tilde{S} \setminus S^{\star}\}$ and $\{A_{ij} : (i,j) \in S^{\star} \setminus \tilde{S}\}$ are each composed of independent random variables.
By Hoeffding's inequality together with a probability union bound, we have that 
\[
\left(\mathcal{A}_7(A,\tilde{z},z^{\star})\right)^c \mbox{ holds w.e.p.}
\]
Equivalently,
\be\label{StatisticsDifference}
I_1^{2} = O\left(|\Delta S|\log(n)^2\right)\quad  \mbox{w.e.p.}\nonumber
\ee
Using a similar argument to bound $I_2$ yields 
\be \label{StatisticsDifference2} 
I_2^{2} = O\left(|\Delta S|\log(n)^2\right)\quad  \mbox{w.e.p.}
\ee 
From \eqref{sumA_ij}-\eqref{StatisticsDifference2}, we obtain that w.e.p.
\[
\left(\sum_{i\in \tilde{z}^{-1}(a)}\sum_{j\in \tilde{z}^{-1}(b)}A_{ij}-\sum_{i\in (z^{\star})^{-1}(a)}\sum_{j\in (z^{\star})^{-1}(b)}A_{ij}\right)^2 = O(|\Delta S|\log(n)^2)+O(|\Delta S|^2 n^{-\frac{2\alpha}{1+\alpha}+4\alpha \delta}).
\]

Since the vertex sets of cluster $a$ corresponding to the partition functions $\tilde{z}$ and $z^{\star}$ differ only in the $\mathfrak{D}$ boundary vertices at each end, we can derive
\be\label{DeltaS-Case1}
|\Delta S| = O\left(2\mathfrak{D}q + 2\mathfrak{D}(q-2\mathfrak{D}) \right)= O(n^{\frac{2\alpha}{1+\alpha}+5\delta}),
\ee
from which we obtain the inequality \eqref{MainContribution} immediately, and can conclude with Inequality \eqref{IneqCase1WhatWeWant}.

\item Case 2: $\tilde{z}(i)=z^{\star}(i)$ and $\tilde{z}(j)\neq z^{\star}(j)$, or vice versa. We note that 
\begin{align*}
|L_2| \leq n^{(r)} \cdot m\cdot 2\mathfrak{D},
\end{align*}
which implies $|L_2|=O(n^{2-\delta})$.
Without loss of generality, we assume that $\tilde{z}(i)=z^{\star}(i)=a$ and $\tilde{z}(j)=b\neq z^{\star}(j)$. By Theorem \ref{theorem-eras}, we know that $|z^{\star}(j)-b| \leq 1$. It then follows that
\begin{align*}
\bar{A}_{\tilde{z}(i)\tilde{z}(j)}(\tilde{z})&-\bar{A}_{z^{\star}(i)z^{\star}(j)}(z^{\star})\\
&\leq \frac{1}{q^2}\left(\sum_{i\in \tilde{z}^{-1}(a)}\sum_{j\in \tilde{z}^{-1}(b)}A_{ij}-\sum_{i\in (z^{\star})^{-1}(a)}\sum_{j\in (z^{\star})^{-1}(b\pm 1)}A_{ij}\right)(1+o(1)).
\end{align*}
Using the same notation and arguments as in Case 1, we can prove
\be 
\left(\sum_{i\in \tilde{z}^{-1}(a)}\sum_{j\in \tilde{z}^{-1}(b)}A_{ij}-\sum_{i\in (z^{\star})^{-1}(a)}\sum_{j\in (z^{\star})^{-1}(b\pm 1)}A_{ij}\right)^2  =  O(n^{\frac{2\alpha}{1+\alpha}+4 \delta (\alpha+3)})
\nonumber
\ee
by replacing \eqref{DeltaS-Case1} with
\[
|\Delta S| = O(\mathfrak{D}q+q^2)=O(n^{\frac{2\alpha}{1+\alpha}+6\delta}).
\]
Consequently,
\be 
\nonumber \sum_{(i,j)\in L_2} \left(\tilde{\theta}^{(r)}_{ij}-\check{\theta}_{ij}\right)^2 &=\sum_{(i,j)\in L_2} \left(\bar{A}_{\tilde{z}(i)\tilde{z}(j)}(\tilde{z})-\bar{A}_{z^{\star}(i)z^{\star}(j)}(z^{\star})\right)^2\\
\label{IneqCase2WhatWeWant}&= \frac{|L_2|}{q^4}\, O\left(n^{\frac{2\alpha}{1+\alpha}+4 \delta (\alpha+3)} \right)=  O(n^{\frac{2}{1+\alpha}+4\alpha\delta - \delta}).
\ee

\item Case 3: $\tilde{z}(i)\neq z^{\star}(i), \tilde{z}(j)\neq z^{\star}(j)$. This case has the worst point-wise estimates, though the set $L_{3}$ is small. We note that
\[
|L_3|\leq m^2(2\mathfrak{D})^2=O(m^2 \mathfrak{D}^2)=O(n^{2-2\delta}).
\]
By arguments similar to those above, we obtain
\begin{equation*}
\left(\sum_{i\in \tilde{z}^{-1}(a)}\sum_{j\in \tilde{z}^{-1}(b)}A_{ij}-\sum_{i\in (z^{\star})^{-1}(a)}\sum_{j\in (z^{\star})^{-1}(b\pm 1)}A_{ij}\right)^2  =  O(n^{\frac{2\alpha}{1+\alpha}+4 \delta (\alpha+3)})
\end{equation*}
by replacing \eqref{DeltaS-Case1} with
\[
|\Delta S| = O(q^2)=O(n^{\frac{2\alpha}{1+\alpha}+6\delta}),
\]
and
\be 
\nonumber \sum_{(i,j)\in L_3} \left(\tilde{\theta}^{(r)}_{ij}-\check{\theta}_{ij}\right)^2 &=\sum_{(i,j)\in L_3} \left(\bar{A}_{\tilde{z}(i)\tilde{z}(j)}(\tilde{z})-\bar{A}_{z^{\star}(i)z^{\star}(j)}(z^{\star})\right)^2\\
\label{IneqCase3WhatWeWant} &= \frac{|L_3|}{q^4}\, O\left( n^{\frac{2\alpha}{1+\alpha}+4\delta(\alpha+3)}\right) =  O(n^{\frac{2}{1+\alpha}+4 \alpha \delta - 2 \delta}).
\ee
\end{enumerate}

Combining the estimates \eqref{IneqCase1WhatWeWant}, \eqref{IneqCase2WhatWeWant} and \eqref{IneqCase3WhatWeWant} for Cases 1-3, we conclude that
\be \label{IneqTerm1MainConc}
\|\tilde{\theta}^{(r)}-\check{\theta}\|_2^2 = \sum_{k=1}^3 \sum_{(i,j)\in L_k} \Big(\tilde{\theta}^{(r)}_{ij} -\check{\theta}_{ij}\Big)^2 = O(n^{\frac{2}{1+\alpha}+4 \alpha \delta -\delta}).
\ee

The second term on the RHS of \eqref{main_inequality} is written as
\be
\nonumber \|\check{\theta}-\theta^{\star}\|_2^2 &=\sum_{i,j\in V^{(r)}_c} (\check{\theta}_{ij}-\theta^{\star}_{ij})^2\\
\nonumber &= \sum_{a,b\in [m]}|(z^{\star})^{-1}(a)| |(z^{\star})^{-1}(b)| (\bar{A}_{ab}(z^{\star})-\bar{\theta}_{ab}(z^{\star}))^2\\
\label{max_estimate} &\leq  \max_{z\in Z_{V^{(r)}_c,m}} \sum_{a,b\in [m]} |z^{-1}(a)| |z^{-1}(b)| (\bar{A}_{ab}(z)-\bar{\theta}_{ab}(z))^2.
\ee

By the argument in the proof of Lemma 4.1 in \cite{gao_rog}, we can obtain an upper bound for \eqref{max_estimate}:
\be \label{IneqTerm2MainConc}
\|\check{\theta}-\theta^{\star}\|_2^2 = O( m^2+n \log m ),
\ee
which is of order $O(n^{\frac{2}{1+\alpha}})$ for sufficiently large $n$ such that $\log(n) < n^{\frac{1-\alpha}{1+\alpha}}$.

It remains to bound the third term on the RHS of \eqref{main_inequality}. In view of Assumption \ref{assumption:Holder},  we have for $i\in (z^{\star})^{-1}(a)$ and $j\in (z^{\star})^{-1}(b)$ with $a\neq b$,
\begin{eqnarray*}
&& \left|w(U_i,U_j)- \bar{\theta}_{ab}(z^{\star})\right| \\
&=& \left|w(U_i,U_j)- \frac{1}{|(z^{\star})^{-1}(a)| |(z^{\star})^{-1}(b)|}\sum_{x\in (z^{\star})^{-1}(a)}\sum_{y\in (z^{\star})^{-1}(b)}w(U_x,U_y)\right|\\
&\leq & \frac{1}{|(z^{\star})^{-1}(a)| |(z^{\star})^{-1}(b)|} \sum_{x\in (z^{\star})^{-1}(a)}\sum_{y\in (z^{\star})^{-1}(b)}\left|w(U_i,U_j)- w(U_x,U_y)\right|\\
&= & \max_{\substack{i,x\in (z^{\star})^{-1}(a),\\ j,y\in (z^{\star})^{-1}(b)}}\left(O\left(|U_i-U_x|\right)^{\alpha}+O\left(|U_j-U_y|\right)^{\alpha}\right) = O(m^{-\alpha}),
\end{eqnarray*}
which implies that
\be \label{IneqTerm3MainConc}
\|\theta^{\star}-\theta\|_2^2 = O(n^2 m^{-2\alpha}) = O(n^{\frac{2}{1+\alpha}+6\alpha \delta}).
\ee

Combining Inequalities \eqref{IneqTerm1MainConc}, \eqref{IneqTerm2MainConc} and \eqref{IneqTerm3MainConc} with Inequality \eqref{main_inequality}, we have our main estimate:
\[
\|\tilde{\theta}^{(r)}-\theta\|_2^2 = O(n^{\frac{2}{1+\alpha}+6\alpha \delta }), \quad \mbox{for $r\in [3]$}.
\]

Finally, we combine the estimates for different indices $r \in [3]$. Recalling \eqref{tilde-theta},
\begin{align*}
\nonumber &\qquad \sum_{i,j\in [n]} (\tilde{\theta}_{ij}-\theta_{ij})^2 \\
\nonumber &\leq \sum_{i,j\in [n]}\sum_{r,\ell \in [3]} \mathbf{1}_{\{i\in V^{(r)}, j\in V^{(\ell)}\}} \left(\frac{\sum_{k\in [3]\setminus \{r,\ell\}}\tilde{\theta}^{(k)}_{ij}}{\left|[3]\setminus \{r,\ell\}\right|}-\theta_{ij}\right)^2\\
\nonumber & = \sum_{r,\ell\in [3]} \sum_{i\in V^{(r)}, j\in V^{(\ell)}}\left(\frac{\sum_{k\in [3]\setminus \{r,\ell\}}\tilde{\theta}^{(k)}_{ij}}{\left|[3]\setminus \{r,\ell\}\right|}-\theta_{ij}\right)^2\\
\nonumber &= \sum_{r\neq \ell \in [3]} \sum_{i\in V^{(r)}, j\in V^{(\ell)}} \left(\tilde{\theta}^{([3]\setminus \{r,\ell\})}_{ij}-\theta_{ij}\right)^2 + \sum_{r\in [3]}\sum_{i,j \in V^{(r)}} \left(\frac{\sum_{k\in [3]\setminus \{r\}}\tilde{\theta}^{(k)}_{ij}}{2}-\theta_{ij}\right)^2\\
\nonumber &\leq \sum_{r\neq \ell \in [3]} \sum_{i\in V^{(r)}, j\in V^{(\ell)}} \left(\tilde{\theta}^{([3]\setminus \{r,\ell\})}_{ij}-\theta_{ij}\right)^2 + \sum_{r\in [3]}\sum_{i,j \in V^{(r)}} \frac{1}{2}\sum_{k\in [3]\setminus \{r\}}\left(\tilde{\theta}^{(k)}_{ij}-\theta_{ij}\right)^2\\
\label{theta-4} & = O(n^{\frac{2}{1+\alpha}+6\alpha \delta}).
\end{align*}
Thus, we conclude
\[
\frac{1}{n^2}\|\tilde{\theta}-\theta\|_2^2 = O (n^{-\frac{2\alpha}{1+\alpha}+6\alpha \delta}).
\]
\end{proof}

\subsection{Application to Property Testing}\label{SecPT}

Up to this point, we have written algorithms \textit{assuming} that our random graph $G$ was sampled from a graphon $w$ satisfying certain assumptions. In this section, we write down a way to \textit{test} if a random graph $G$ has this property. The test is based on the $\Lambda$ statistic introduced in \cite{LambdaDef24}, itself based on the $\Gamma$ statistic introduced in \cite{chuangpishit2015linear}.

Before defining the $\Lambda$ statistic, we need one piece of notation: for sets $A,B \subset [0,1]$, we write $A \leq B$ to mean that $a \leq b$ for all $a \in A, b \in B$. For a graphon $w$, the associated $\Lambda$-statistic (from Definition 4 of \cite{LambdaDef24}) is:
\be 
\Lambda(w) &=  \frac{1}{2} \sup_{A \leq B \leq C, |A|=|B|=|C|}\left( \int_{x \in A} \int_{y \in C} w(x,y) dx dy - \int_{x \in B} \int_{y \in C} w(x,y) dx dy \right)  \nonumber\\ 
&+ \frac{1}{2} \sup_{A \leq B \leq C, |A|=|B|=|C|}\left( \int_{x \in A} \int_{y \in C} w(x,y) dx dy - \int_{x \in A} \int_{y \in B} w(x,y) dx dy \right).\nonumber
\ee 
We can immediately see that $\Lambda(w) \geq 0$ for all $w$. By Proposition 5 of \cite{LambdaDef24}, in fact $\Lambda(w) = 0$ if and only if $w$ is Robinson. This suggests using a good approximation $\hat{\Lambda}(G)$ of $\Lambda$ as a test statistic: if the approximation is good enough (for some class of Robinson graphons), we can \textit{reject} the hypothesis that $G \sim w$ for some $w$ in that class whenever we observe a value of $\hat{\Lambda}$ that is large. Our main result of this section, Theorem \ref{Thm_Test_Power}, says that the test statistic computed by Algorithm \ref{alg:graphon-test} has this property:

\begin{theorem} \label{Thm_Test_Power}
Assume that graphon $w$ satisfies Assumptions~\ref{DecayRateAssume} and~\ref{U-r-relationship}. Fix an integer $n \in \mathbb{N}$ and let $G \sim w$ on vertex set $V=[n]$. Let $\alpha >0$ be the decay rate of $w$ and let $\gamma >0$ be the precision parameter of coarse algorithm $P$. Fix $\delta >0$ and let $\epsilon = \epsilon_0(\delta)$ as defined in \eqref{epsilon_0}. For a fixed $\beta \in (0,1)$, set $\mu = n^{-\beta}$. Let $\hat{\Lambda}(G)$ be calculated as in Algorithm~\ref{alg:graphon-test} with coarse algorithm $P$, accuracy $\delta$, size scale $\mu$, and input graph $G$. Then w.e.p.
\be \label{IneqCompTestStat}
\hat{\Lambda}(G) = O(n^{-1-\beta}\log(n)^2) + O\left(\min\left\{n^{-2\beta},\, n^{-\beta-\frac{1}{1+\alpha}+2\delta}\log(n)\right\}\right).
\ee
If Assumption~\ref{assumption:Holder} also holds, then w.e.p. the sharper estimate
\be 
\label{IneqCompTestStatHolder}  \hat{\Lambda}(G) = O( n^{-1-\beta}\log(n)^2) + O\left(n^{-2\beta-\frac{\alpha}{1+\alpha}+2\alpha\delta}\log(n)^\alpha\right)
\ee
also holds. The constants may depend on the graphon and on $\delta$, but not on $n$.
\end{theorem}

The scale $\mu=n^{-\beta}$ controls the largest intervals considered by the test.

We give the algorithm, then prove the result. Recall that for a set $S \subset V$ and an ordering $\sigma$ on $S$,
\[
\mathcal{I}(S,\sigma)
=
\{ I(S,\sigma,k,l) : k,l \in S \}
\]
denotes the collection of intervals induced by $\sigma$, where
\[
I(S,\sigma,k,l)
=
\{ s \in S : \sigma(k) < \sigma(s) < \sigma(l) \}.
\]
For a given set $V$, rank function $\sigma$ on $V$, and size scale $\mu$, define $\mathcal{G}(V,\sigma,\mu)$ to be the set of triples of disjoint intervals $A,B,C\in \mathcal{I}(V,\sigma)$ satisfying
\be 
1\leq |A| = |B| &= |C| \leq \mu\, |V|, \nonumber\\
\max_{a \in A} \sigma(a) &< \min_{b \in B} \sigma(b), \nonumber\\
\max_{b \in B} \sigma(b) &< \min_{c \in C} \sigma(c). \nonumber
\ee 
The algorithmic procedure is described in Algorithm~\ref{alg:graphon-test}.

\begin{algorithm}[h]
\caption{Robinson Test Statistic}\label{alg:graphon-test}
\begin{algorithmic}[1]
\algrenewcommand\alglinenumber[1]{}
\State \textbf{parameters:} Size scale $\mu$, decay rate $\alpha \in [0,1]$, accuracy $\delta>0$ and coarse algorithm $P$ with precision parameter $\gamma$. 
\State \textbf{input:} Observed graph $G = (V,E)$ with adjacency matrix $A = \{A_{ij}: i,j \in [n]\}$.
\State \textbf{output:} Test statistic $\hat{\Lambda}$.
\algrenewcommand\alglinenumber[1]{\scriptsize #1}
\setcounter{ALG@line}{0}
\State Select a partition of $V$ into two subsets $V^{(1)}$, $V^{(2)}$ of near-equal sizes $|\, |V^{(1)}| - |V^{(2)}|\, | \leq 1$, uniformly at random. 
\State Construct $G_{\mathrm{train}}$ from $\{ A_{ij} : i, j \in V^{(1)} \}$ and run Algorithm~\ref{alg:eras} on $G_{train}$ with $n^{(1)}=|V^{(1)}|$ and parameters $k$, $\{p_i\}_{i=1}^k$, $\{d_i\}_{i=1}^k$, $C_1$, $C_2$ and $C_3$ specified in \eqref{parameter3} to obtain the ordering $\sigma_{train}$ on $V^{(1)}$.
\State Run Algorithm~\ref{alg:erss} on the full graph $G$ with input sets $V^{(1)} \subset V$, the ordering $\sigma_{train}$ on $V^{(1)}$ and parameters $C'_1$, $C'_2$ and $C'_3$ specified in \eqref{parameter3'} to produce the extended ordering $\sigma_{extend}$ on $V$.
\State Define $\sigma$ by $\sigma(i)= |\{j\in V^{(2)}: \sigma_{extend}(j) \leq \sigma_{extend}(i)\}|$ for $i\in V^{(2)}$.
\For{$(A,B,C) \in \mathcal{G}(V^{(2)}, \sigma, \mu)$}
\State Calculate
\[
\hat{\Lambda}_{1}(A,B,C) = \frac{1}{|V^{(2)}|^{2}}\left(\sum_{a \in A, c \in C} A_{ac} -  \sum_{b \in B, c \in C} A_{bc}\right)
\]
and
\[
\hat{\Lambda}_{2}(A,B,C) = \frac{1}{|V^{(2)}|^{2}} \left(\sum_{a \in A, c \in C} A_{ac}
-  \sum_{a \in A, b \in B} A_{ab}\right).
\]
\EndFor
\State Return 
\[
\frac{1}{2}\left(\max_{(A,B,C) \in  \mathcal{G}(V^{(2)}, \sigma, \mu) } \hat{\Lambda}_{1}(A,B,C) + \max_{(A,B,C) \in  \mathcal{G}(V^{(2)}, \sigma, \mu) } \hat{\Lambda}_{2}(A,B,C)\right).
\]
\end{algorithmic}
\end{algorithm}

For any sets $A, B \subseteq S$, use the convention $A_{ii}=\theta_{ii}=0$ and $\theta_{ij}=w(U_i,U_j)$ for $i\neq j$. The expected value of the corresponding adjacency-matrix sum is
\be\label{EqDefWTest}
W(A, B)
=
\sum_{i \in A} \sum_{j \in B} \theta_{ij}
=
\mathbb{E}\left[\sum_{i \in A}\sum_{j \in B} A_{ij}\,\middle|\,(U_\ell)_{\ell\in V}\right].
\ee
We now define
\be\label{EqDefA8Test}
\mathcal{A}_8(S,\sigma)
=
\bigcup_{A,B \in \mathcal{I}(S,\sigma)}
\left\{
\left| \sum_{a \in A} \sum_{b \in B} A_{ab} - W(A,B) \right|
>
\sqrt{W(A,B)}\,\log (n) + \log (n)^2
\right\}.
\ee

\begin{lemma}\label{A6-lemma}
Let $V^{(2)}$ be the subset defined in Step~1, and let $\sigma$ be the rank function on $V^{(2)}$ obtained in Step~4 of Algorithm~\ref{alg:graphon-test}. Then, w.e.p., $(\mathcal{A}_8(V^{(2)},\sigma))^c$ holds.
\end{lemma}

\begin{proof}
Let $N=|V^{(2)}|$, and let $\mathcal{H}$ be generated by the latent positions, the random split in Step~1 of Algorithm~\ref{alg:graphon-test}, the auxiliary randomness in Algorithms~\ref{alg:eras} and~\ref{alg:erss}, and the edge variables used in Steps~2--4 of Algorithm~\ref{alg:graphon-test} to construct $\sigma$. By the sample splitting described before Algorithm~\ref{alg:erss}, the edge variables $\{A_{uv}:u<v,\ u,v\in V^{(2)}\}$ are not used in the construction of $\sigma$ and are left for Step~6 of Algorithm~\ref{alg:graphon-test}. Thus, once $\mathcal{H}$ is fixed, $\sigma$ and the intervals in $\mathcal{I}(V^{(2)},\sigma)$ are fixed, while these edge variables are independent Bernoulli variables with parameters $w(U_u,U_v)$.

Fix $A,B\in \mathcal{I}(V^{(2)},\sigma)$ and write
\[
T(A,B)=\sum_{a\in A}\sum_{b\in B} A_{ab}.
\]
By \eqref{EqDefWTest}, $\mathbb{E}[T(A,B)\mid \mathcal{H}]=W(A,B)$. Since the graph is undirected, the two terms $A_{uv}$ and $A_{vu}$ may both appear in $T(A,B)$. Combining such terms, $T(A,B)-W(A,B)$ is a sum over unordered pairs $\{u,v\}\subset V^{(2)}$ of independent terms, each bounded by $2$, with total variance at most $2W(A,B)$. Bernstein's inequality therefore gives, for a universal constant $c>0$,
\[
\mathbb{P}\left(
|T(A,B)-W(A,B)|>
\sqrt{W(A,B)}\,\log(n)+\log(n)^2
\,\middle|\,\mathcal{H}
\right)
\leq 2e^{-c\log(n)^2}.
\]
By \eqref{EqDefA8Test} and $|\mathcal{I}(V^{(2)},\sigma)|\leq N^2$, a union bound over at most $N^4\leq n^4$ pairs of intervals gives
\[
\mathbb{P}\left(\mathcal{A}_8(V^{(2)},\sigma)\mid \mathcal{H}\right)
\leq 2n^4 e^{-c\log(n)^2}=n^{-\Omega(\log n)}.
\]
Taking expectations over $\mathcal{H}$ proves that $(\mathcal{A}_8(V^{(2)},\sigma))^c$ holds w.e.p.
\end{proof}

\begin{proof} [Proof of Theorem \ref{Thm_Test_Power}]
Let $N=|V^{(2)}|$, so $N=\Theta(n)$, and let $\sigma_{true}$ denote the true rank function on $V^{(2)}$. Set
\[
\rho_N=N^{-\frac{1}{1+\alpha}+2\delta}\log N.
\]
By Theorem~\ref{theorem-eras} in Step~2, Theorem~\ref{theorem-erss} in Step~3, and Lemma~\ref{PreciseLevel-Error}, w.e.p.
\be\label{EqRankErrorTest}
D=\max_{i\in V^{(2)}}|\sigma(i)-\sigma_{true}(i)|,
\qquad \frac{D}{N}=O(\rho_N).
\ee
We work on the event in \eqref{EqRankErrorTest} and on the event $(\mathcal{A}_8(V^{(2)},\sigma))^c$ from Lemma~\ref{A6-lemma}. When proving Inequality \eqref{IneqCompTestStatHolder}, we also use the analogous empirical-spacing estimate for the random half $V^{(2)}$, obtained by the same argument as for $\mathcal{A}_1(p)$: vertices whose true ranks in $V^{(2)}$ differ by $O(D)$ have latent positions differing by $O(D/N)$, w.e.p.

Fix $(A,B,C)\in \mathcal{G}(V^{(2)},\sigma,\mu)$, and write $s=|A|=|B|=|C|\leq \mu N$. For each $Z\in\{A,B,C\}$, choose the integer $r_Z$ such that
\[
Z=\{i\in V^{(2)}:r_Z<\sigma(i)\leq r_Z+s\},
\]
and define the true-order interval
\[
Z^\star=\{i\in V^{(2)}:r_Z<\sigma_{true}(i)\leq r_Z+s\}.
\]
The ordering conditions in the definition of $\mathcal{G}(V^{(2)},\sigma,\mu)$ imply $A^\star\leq B^\star\leq C^\star$ in the true order. By \eqref{EqRankErrorTest}, a vertex can belong to exactly one of $Z$ and $Z^\star$ only if its estimated or true rank is within $D$ of an endpoint of the interval. Therefore, for each $Z\in\{A,B,C\}$,
\be\label{EqSymDiffTest}
|Z\triangle Z^\star|\leq \min\{4D,2s\}=O(\min\{s,D\}),
\ee
where $Z\triangle Z^\star=(Z\setminus Z^\star)\cup (Z^\star\setminus Z)$. Since $w$ is Robinson,
\be\label{EqRobinsonTest}
W(A^\star,C^\star)\leq W(B^\star,C^\star),\qquad
W(A^\star,C^\star)\leq W(A^\star,B^\star).
\ee

We first bound $\hat{\Lambda}_1(A,B,C)$. On $(\mathcal{A}_8(V^{(2)},\sigma))^c$, for any intervals $X,Y\in \mathcal{I}(V^{(2)},\sigma)$ of size at most $s$, Equation~\eqref{EqDefA8Test} gives
\be\label{EqA8UseTest}
\left|\sum_{x\in X}\sum_{y\in Y}A_{xy}-W(X,Y)\right|
\leq \sqrt{W(X,Y)}\log(n)+\log(n)^2
=O(s\log(n)^2),
\ee
where we used $W(X,Y)\leq |X||Y|\leq s^2$. Applying \eqref{EqA8UseTest} to $(A,C)$ and $(B,C)$,
\be
\hat{\Lambda}_1(& A, B,C)\\
&\leq \frac{1}{N^2}O(s\log(n)^2)
+\frac{1}{N^2}\left(W(A,C)-W(B,C)\right)\notag\\
&\leq O(n^{-1-\beta}\log(n)^2)\notag\\
& \quad +\frac{1}{N^2}\Big(|W(A,C)-W(A^\star,C^\star)|+|W(B,C)-W(B^\star,C^\star)|\Big),
\label{EqLambda1DecompTest}
\ee
where the second inequality uses \eqref{EqRobinsonTest}, $s\leq \mu N$, $\mu=n^{-\beta}$, and $N=\Theta(n)$.

The crude bound $0\leq \theta_{ij}\leq 1$ gives, for either pair $(X,Y)=(A,C)$ or $(B,C)$,
\be
|W(X,Y)-W(X^\star,Y^\star)|
&\leq |W(X,Y)-W(X^\star,Y)|+|W(X^\star,Y)-W(X^\star,Y^\star)|\notag\\
&\leq |X\triangle X^\star|\,|Y|+|Y\triangle Y^\star|\,|X^\star|\notag\\
&=O(s\min\{s,D\}),
\label{EqCrudeWTest}
\ee
by \eqref{EqSymDiffTest}. After dividing by $N^2$, \eqref{EqCrudeWTest} contributes at most
\[
O\left(\frac{s}{N}\min\left\{\frac{s}{N},\frac{D}{N}\right\}\right)
=O\left(\mu\min\left\{\mu,\rho_N\right\}\right)
=O\left(\min\left\{n^{-2\beta},\,n^{-\beta-\frac{1}{1+\alpha}+2\delta}\log(n)\right\}\right).
\]
Together with \eqref{EqLambda1DecompTest}, this gives the contribution in \eqref{IneqCompTestStat} for $\hat{\Lambda}_1(A,B,C)$.

Assume now that Assumption~\ref{assumption:Holder} also holds. For each $Z\in\{A,B,C\}$, match the $t$-th vertex of $Z$ in the $\sigma$-ordering to the $t$-th vertex of $Z^\star$ in the $\sigma_{true}$-ordering. By \eqref{EqRankErrorTest}, matched vertices differ by $O(D)$ true ranks in $V^{(2)}$. The empirical-spacing estimate in Inequality \eqref{EqRankErrorTest} and the immediately-following paragraph gives latent-position differences $O(D/N)=O(\rho_N)$ for matched vertices. Hence, for either pair $(X,Y)=(A,C)$ or $(B,C)$,
\be
|W(X,Y)-W(X^\star,Y^\star)|
&\leq \sum_{x\in X}\sum_{y\in Y}\left|w(U_x,U_y)-w(U_{x^\star},U_{y^\star})\right|\notag\\
&\leq \sum_{x\in X}\sum_{y\in Y}\left(M_2|U_x-U_{x^\star}|^\alpha+M_2|U_y-U_{y^\star}|^\alpha\right)\notag\\
&=O(s^2\rho_N^\alpha),
\label{EqHolderWTest}
\ee
where $x^\star$ and $y^\star$ denote the matched vertices, and the second line uses Assumption~\ref{assumption:Holder} and symmetry of $w$. After normalization, \eqref{EqHolderWTest} contributes at most
\[
O(\mu^2\rho_N^\alpha )
=O\left( n^{-2\beta-\frac{\alpha}{1+\alpha}+2\alpha\delta}\log(n)^\alpha\right).
\]
Together with \eqref{EqLambda1DecompTest}, this gives the contribution in \eqref{IneqCompTestStatHolder} for $\hat{\Lambda}_1(A,B,C)$.

The proof for $\hat{\Lambda}_2(A,B,C)$ is the same, using the pairs $(A,C)$ and $(A,B)$ and the second inequality in \eqref{EqRobinsonTest}. The bounds above are uniform over $(A,B,C)\in \mathcal{G}(V^{(2)},\sigma,\mu)$, so taking the two maxima in Algorithm~\ref{alg:graphon-test} and averaging proves \eqref{IneqCompTestStat} and \eqref{IneqCompTestStatHolder}.
\end{proof}

\begin{remark}[Scale of the test statistic and detectable alternatives]\label{RemHowGoodTest}
Theorem~\ref{Thm_Test_Power} bounds
$\hat{\Lambda}(G)$ as a sum of two terms, which we think of as corresponding to a sampling term and an ordering-mismatch term. Under Assumption~\ref{assumption:Holder}, the error is on the scale:
\[
R_{\mathrm{null}}(n,\beta,\delta)
=
n^{-1-\beta}\log(n)^2
+
n^{-2\beta-\frac{\alpha}{1+\alpha}+2\alpha\delta}\log(n)^\alpha .
\]

Is that roughly the correct error, and is it small enough to let us detect non-Robinson graphons? Towards answering the first question, note that an Erd\HungarianH{o}s--R\'enyi graph is sampled from a constant Robinson graphon. For a fixed interval triple with $|A|=|B|=|C|\asymp n^{1-\beta}$, the normalized edge-count fluctuation is on
the scale $n^{-1-\beta}$. Maximizing over polynomially many interval triples adds only logarithmic factors, so the sampling term in Theorem~\ref{Thm_Test_Power} has the correct dependence on $\beta$, up to
logarithmic factors.\footnote{The implied constants in the upper and lower bounds need not be the same.}

Towards answering the second question, suppose there are intervals $I_A<I_B<I_C$ of length comparable to
$\mu=n^{-\beta}$ and a constant $\eta>0$ such that
\[
w(x,z)>\max\{w(x,y),w(y,z)\}+\eta
\]
for all $x\in I_A$, $y\in I_B$, and $z\in I_C$. The associated signal in this interval is at least
\[
\eta\mu^2\asymp \eta n^{-2\beta}.
\]
In particular, for all $\alpha > 0$ and all $\delta$ sufficiently small this is detectable by our test.
\end{remark}

We conclude with some basic comments on Algorithm~\ref{alg:graphon-test} and the associated test:

\begin{enumerate}
    \item The run-time of Algorithm~\ref{alg:graphon-test} is polynomial in $n$, since there are at most $O(n^{6})$ ordered interval triples in the ordered set $[n]$. This is a large improvement on computing the full optimization procedure in \cite{LambdaDef24}, but it is still rather slow. In practice, one could substantially cut down on this computation by only considering intervals of the form $I_{k} = \{ i \, : \, Ck \leq i \leq C(k+1)\}$ for suitably large step-size $C$ and by using sketching or other subsampling methods.
    \item Since Algorithm~\ref{alg:graphon-test} depends on the ordering algorithm, it takes the decay rate $\alpha$ as a parameter. It would be straightforward to replace Steps 2--3 of Algorithm~\ref{alg:graphon-test} by calls to any other seriation algorithm, and one would expect a result similar to Theorem~\ref{Thm_Test_Power} to hold with the corresponding ordering rate replacing the rate supplied by Theorem~\ref{theorem-eras}.
\end{enumerate}

\bibliographystyle{alpha}
\bibliography{ESCBib}

\appendix
\section{Bijections Between Mislabeled Points}\label{sec:appendix-proofs}

We use notation from the proof of Theorem~\ref{ourTheorem}.

\begin{lemma}\label{1to1-appendix}
Suppose that $\sigma$ has error at most $\mathfrak{D}$, where $2\mathfrak{D}<q$. Fix $a,b\in[m]$, and define
\[
\widetilde S_{ab}=\widetilde S_1(a)\times \widetilde S_1(b),\qquad
S^\star_{ab}=S^\star_1(a)\times S^\star_1(b),
\]
where $\widetilde S_1(a)=\tilde z^{-1}(a)$ and $S^\star_1(a)=(z^\star)^{-1}(a)$. For $n$ sufficiently large, there exists a bijection between pairs $(\tilde{i},\tilde{j})\in \widetilde S_{ab}\setminus S^\star_{ab}$ and pairs $(i,j)\in S^\star_{ab}\setminus \widetilde S_{ab}$. This bijection can be chosen so that the matched pairs satisfy
\[
|U_i - U_{\tilde{i}}| = O\left(\frac{\mathfrak{D}}{n}\right)
\]
and
\[
|U_j - U_{\tilde{j}}| = O\left(\frac{\mathfrak{D}}{n}\right).
\]
\end{lemma}

\begin{proof}
For any $a\in[m]$, write
\begin{align*}
\widetilde S_1(a)&=\{i\in V^{(r)}_c:(a-1)q<\sigma(i)\leq aq\},\\
S^\star_1(a)&=\{i\in V^{(r)}_c:(a-1)q<\sigma_{true}(i)\leq aq\}.
\end{align*}
Define the forward and backward $\mathfrak{D}$-neighborhoods of the block boundaries by
\[
\widetilde F_1(a)=\{i\in V^{(r)}_c:(a-1)q<\sigma(i)\leq (a-1)q+\mathfrak{D}\},
\]
\[
F^\star_1(a)=\{i\in V^{(r)}_c:(a-1)q<\sigma_{true}(i)\leq (a-1)q+\mathfrak{D}\},
\]
\[
\widetilde B_1(a)=\{i\in V^{(r)}_c:aq-\mathfrak{D}<\sigma(i)\leq aq\},
\]
and
\[
B^\star_1(a)=\{i\in V^{(r)}_c:aq-\mathfrak{D}<\sigma_{true}(i)\leq aq\}.
\]
Since $\sigma$ and $\sigma_{true}$ differ by at most $\mathfrak{D}$ in rank, the only discrepancies between $\widetilde S_1(a)$ and $S^\star_1(a)$ occur in these boundary neighborhoods. Matching vertices in increasing true rank gives a bijection from $\widetilde S_1(a)\setminus S^\star_1(a)$ to $S^\star_1(a)\setminus \widetilde S_1(a)$, with the identity map on $\widetilde S_1(a)\cap S^\star_1(a)$, such that matched vertices have true ranks differing by $O(\mathfrak{D})$. The empirical-spacing bound for the latent positions then gives
\[
|U_i-U_{\tilde i}|=O\left(\frac{\mathfrak{D}}{n}\right)
\]
for matched vertices.

Applying this one-dimensional matching in the two coordinates gives a product matching from $\widetilde S_{ab}$ to $S^\star_{ab}$ that is the identity on the intersection. Therefore it restricts to a bijection between $\widetilde S_{ab}\setminus S^\star_{ab}$ and $S^\star_{ab}\setminus \widetilde S_{ab}$, and the coordinatewise latent-position bounds above give the stated estimates.
\end{proof}

\end{document}